\documentclass[12pt]{article}

\usepackage[cp1251]{inputenc}
\usepackage[english,russian]{babel}
\usepackage[tbtags]{amsmath}
\usepackage{amsfonts,amssymb,amsthm,cite}

\usepackage[active]{srcltx}

\usepackage{color}

\usepackage{mathrsfs}

\usepackage{graphicx}

\usepackage[width=160mm, left=20mm, top=10mm, height=245mm, paper=a4paper]{geometry}

\newtheorem{theorem}{\hskip\parindent Theorem}[section]
\newtheorem{lemma}{\hskip\parindent Lemma}[section]

\newtheorem{definition}{\hskip\parindent Definition}[section]
\newtheorem{remark}{\hskip\parindent Remark}[section]

\addto\capionsenglish{}

\def\loc{\text{\rm loc}}

\begin{document}

\noindent {\footnotesize {UDK 517.51}} 
\vspace{5mm}

\begin{center} 
{\bf
\Large Spline\! wavelet\! decomposition\! in\! weighted\! function\! spaces\footnote{The results of \S~4 of the work were performed at Steklov Mathematical Institute of Russian Academy of Sciences under financial support of the Russian Science Foundation (project 19-11-00087). The work of the rest part of the paper was carried out within the framework of the State Tasks of Ministry of Education and Science of Russian Federation for V.A. Trapeznikov Institute of Control Sciences of Russian Academy of Sciences and Computing Center of Far--Eastern Branch of Russian Academy of Sciences, it was also partially supported by the Russian Foundation for Basic Research (project 19--01--00223).}} 
\end{center}\begin{center}
{\bf \large   E. P. Ushakova}
\end{center}
\vskip0.2cm

{\small Battle--Lemari\'{e} wavelet systems of natural orders are established in the paper. The main result of the work is decom\-po\-si\-tion theorem in Besov and Triebel--Lizorkin spaces with local Muckenhoupt weights, which is performed in terms of bases generated by the systems of such a type. Battle--Lemari\'{e} wavelets are splines and suit very well for the study of integration operators. 

Bibliography: 80 items.

}

\parindent 15pt
\section{Introduction}

This paper is a generalisation of \cite{RMC}.
Let $B_{pq}^s(\mathbb{R}^N,w)$ and $F_{pq}^s(\mathbb{R}^N,w)$ be weighted analogs of the unweighted Besov $B_{pq}^s(\mathbb{R}^N)$ and Triebel--Lizorkin spaces $F_{pq}^s(\mathbb{R}^N)$ in Euclidean $N-$space $\mathbb{R}^N$ with $0<p\le\infty$, $0<q\le\infty$, $-\infty<s<\infty$ and weights $w$ of Muckenhoupt type (see \S~\ref{WFS} for the definitions). The main result in \cite{RMC} is a decomposition of $B_{pq}^s(\mathbb{R}^N,w)$ and $F_{pq}^s(\mathbb{R}^N,w)$, which was performed in terms of compactly supported linear combinations of elements of Battle--Lemari\'{e} spline wavelet systems of natural orders. Recall that, for a given $n\in\mathbb{N}$, Battle--Lemari\'{e} scaling $\phi_n^{BL}$ and wavelet $\psi_n^{BL}$ functions generate an orthonormal spline basis in $L_2(\mathbb{R})$ (see \S~\ref{SWS}). But, being not compactly supported, the functions $\phi_n^{BL}$ and $\psi_n^{BL}$, themselves, cannot be effectively used for decomposing elements of weighted function spaces of Besov and Triebel--Lizorkin type. An important localisation property of elements of Battle--Lemari\'{e} systems was established in \cite{RMC} (see also \cite{JMAA}): there exist finite linear combinations of integer shifts of $\phi_n^{BL}$ or $\psi_n^{BL}$ resulting in some new functions $\Phi_n$ or $\Psi_n$ having compact supports in $\mathbb{R}$ and generating a semi--orthogonal Riesz basis. The result in \cite{RMC} uses the systems $\{\Phi_n,\Psi_n\}$ for decomposing elements of weighted smoothness function spaces $B_{pq}^s(\mathbb{R}^N,w)$ and $F_{pq}^s(\mathbb{R}^N,w)$ into linear sums of simple pieces. 

In recent decades, several concepts of decompositions in Besov and Triebel--Lizorkin type spaces became known. Among them there are atomic and molecular {representations} \cite{Rou,Bow1,Bow2}, sub--atomic (quarkonial) \cite{Tr3,Far} and wavelet \cite{HTr2,Tr5} decompositions as well as local means {characterisations} \cite{TBan,WBan}. All these approaches have found their applications. For instance, representations by wavelet bases appeared to be an effective tool for in\-ves\-ti\-ga\-tions of cha\-rac\-te\-ris\-tic numbers of linear operators 
\cite{{HTr1},{HTr1'},{HTr2}, HSc, HSc1, HSc2, HSc3, KLSSpp, KLSS, NU, Skr, Va}. 
The purpose of this work is the characterisation of Besov $B_{pq}^{s,w}(\mathbb{R}^N)$ and Triebel--Lizorkin $F_{pq}^{s,w}(\mathbb{R}^N)$ spaces (see \S~\ref{WFS} for the definitions), with weights $w$ of a more general type than in $B_{pq}^s(\mathbb{R}^N,w)$ and $F_{pq}^s(\mathbb{R}^N,w)$, with help of spline wavelet systems of Battle--Lemari\'{e} type suitable to further study of characteristic values (entropy numbers, approximation numbers, eigenvalues, etc.) of integration and differentiation operators between the above spaces. 

Spline decompositions in Besov and Triebel--Lizorkin type spaces go back to Z. Ciesielski and T. Figiel \cite{Cie,Cie1,CF,CF1,CF2}. Further studies in this context were undertaken in e.g. \cite{Rop,SjStr,G-CK,G-CK1,Bou,ABM-R,Tyu1,Tyu2}. The connection with {multiresolution analysis can be seen in \cite{Tr,B,B1,L}.
Spline wavelets are also discussed in standard books} {on wavelets \cite{Chui,Dau,Me,W,Mal}. One may also consult \cite[\S~2.12.3]{Tr1} and \cite[\S~2.5]{Tr5}. Spline wavelet systems can be orthogonal or not. One of the most popular
representatives of non--orthogonal splines is the Cohen--Daubechies--Feauveau class of biorthogonal wavelets \cite{CDF}. Dropping the orthogonality requirement helps sometimes to impart properties those are absent in orthonormal systems (for instance, compactness of supports, symmetry, etc.).  
}

Scales of unweighted spaces $B_{pq}^s(\mathbb{R}^N)$ and $F_{pq}^s(\mathbb{R}^N)$, accordingly $0<p\le\infty$, $0<q\le\infty$ and $-\infty<s<+\infty$, contain, in particular, classical Sobolev spaces, local Hardy spaces, Bessel potential spaces, H\"{o}lder--Zygmund spaces. History of related studies can be seen in a number of monographs \cite{BIN,Tr1,Tr2,Tr2',Tr5} and series of papers \cite{FJ,FJ'} (see also \cite{ScTr,BN,SiTr,EdTr}). Weighted function spaces of Besov and Triebel--Lizorkin type were investigated in e.g. \cite{HTr1,HTr1',HTr2} for so called admissible weights $w$. In \cite{Sch,Sch'} they were studied with locally regular weights, in {\cite{Bui,Str,LR,Bow1,Bow2,HP,HSc,HST}} --- with $w$ from the Muckenhoupt class. Authors of \cite{R,IS',WBan,WM,IS} dealt with $w$ belonging to local Muckenhoupt {weight class}. In this work we focus also on weighted Besov spaces $B_{pq}^{s,w}(\mathbb{R}^N)$ and Triebel--Lizorkin spaces  $F_{pq}^{s,w}(\mathbb{R}^N)$ with $w$ belonging to some local Muckenhoupt class $\mathscr{A}_p^\loc$, $p\ge 1$ (\S~\ref{rw}).

The paper is organised as follows. In \S~\ref{WFS} we collect information about Muckenhoupt and local Muckenhoupt weight classes and recall definitions of related smoothness function spaces of Besov and Triebel--Lizorkin types. 
\S~\ref{SWS} is devoted to a class of Battle--Lemari\'{e} spline wavelet systems of natural orders $n$. We reproduce briefly {constructions} of scaling $\phi_n$ and wavelet $\psi_n$ functions from this class and establish localisation algorithms resulting in some related to them compactly supported functions $\mathbf{\Phi}_n$ and $\mathbf{\Psi}_n$. Our main theorem  uses the $\mathbf{\Phi}_n$ and $\mathbf{\Psi}_n$ for decomposing Besov and Triebel--Lizorkin spaces $B_{pq}^{s,w}(\mathbb{R}^N)$ and $F_{pq}^{s,w}(\mathbb{R}^N)$ with local Muckenhoupt weights $w$. It is given in \S~\ref{mn} (see Theorem \ref{main}).

Throughout the paper relations of the type $A\lesssim B$ mean that $A\le cB$ with
some constant $c$, independent of $A$ and $B$, but depending, possibly, on number parameters. We write
$A\approx B$ instead of $A\lesssim B \lesssim A$ and $A\simeq B$ instead of $A=cB$.  
We use $\mathbb Z$, $\mathbb N$ and $\mathbb R$ for integers, natural and real numbers, respectively, and $\mathbb{C}$ for the complex plane. By $\mathbb{N}_0$ we denote the set $\mathbb{N}\cup\{0\}$. The both symbols $dx$ and $|\cdot|$ stand for the $N-$dimensional Lebesgue measure. 
The notation $[s]$ is used for the integer part of $s\in\mathbb{R}$, and $s_+=\max\{s,0\}$. 
We put $r':=
{r}/({r-1})$ if $0<r<\infty$ and $r'=1$ for $r=\infty$. The symbol $\hookrightarrow$ will be used for continuous embeddings. We make use of marks $:=$ and
$=:$ for introducing new quantities. We abbreviate $h(\Omega):=\int_\Omega h(x)\, dx$,
where {$\Omega \subset\mathbb{R}^N$} is some bounded measurable set.

\section{Classes of weights and function spaces}\label{WFS}
\subsection{Weight functions (weights)} Let $w$ be a locally  integrable function, positive almost everywhere (a weight) on $\mathbb{R}^N$.
\subsubsection{Classes of admissible and general locally regular weight functions} Let $\mathbb{N}_0^N$, $N\in\mathbb{N}$, be the set of all multi--indices $\gamma=(\gamma_1,\ldots,\gamma_N)$ with $\gamma_i\in\mathbb{N}_0$, and $|\gamma|=\sum_{j=1}^N\gamma_j$. We use the abbreviation $D^\gamma$ for derivatives.
\begin{definition}{\rm(\cite[Chapter~4]{EdTr}, \cite[Definition 2.1]{HTr2}) \label{Def1}
The class $\mathscr{W}^N$ of \textit{admissible weight functions} is the collection of all infinitely differentiable functions $w\colon \mathbb{R}^N\to (0,\infty)$ with the following properties:
\begin{itemize}
  \item[~] (i) for all $\gamma \in \mathbb{N}_0^N$ there exists a positive constant $c_\gamma$ such that
\begin{equation}\label{adm1}
  |D^\gamma w(x)| \le c_\gamma w(x)\quad \textrm{for all } x\in \mathbb{R}^N;
\end{equation}
  \item[~] (ii) there exist constants $c > 0$ and $\alpha\ge 0$ satisfying the condition
\begin{equation}\label{adm2}
  0 < w(x) \le c\,w(y)(1+|x-y|^2)^{\alpha/2}\quad \textrm{for all } x, y \in \mathbb{R}^N.
\end{equation}
\end{itemize}}
\end{definition} 
In particular, the functions $w(x)=\bigl(1+|x|^2\bigr)^{\alpha/2}$ and $w(x)=\bigl(1+\log(1+|x|^2)\bigr)^\alpha$ with $\alpha\not=0$ belong to the class $\mathscr{W}^N$ of admissible weights.

If a function
$w$ 
 in Definition \ref{Def1} satisfies \eqref{adm1} and, instead of \eqref{adm2}, the following exponential growth condition
\begin{equation*}
  0 < w(x) \le c\,w(y)\exp\bigl(c|x-y|^\beta\bigr)\quad \textrm{for all } x, y \in \mathbb{R}^N\quad\textrm{and some} \quad 0<\beta\le 1,
\end{equation*} then it is called \textit{a general locally regular weight}. Any admissible weight is locally regular. But, for instance, the weight $w(x)=\exp\bigl(|x|^\beta\bigr)$, $0<\beta\le 1$, is locally regular but is not admissible in the sence of Definition \ref{Def1} \cite{Sch,Sch',Ma}.

\subsubsection{Class of Muckenhoupt weights $\mathscr{A}_\infty$}\label{rw}  To describe Muckenhoupt class $\mathscr{A}_\infty$ we appeal to \cite{M,M1,M2,T,S,FS,HP}. 

Let $Q$ be a cube $Q\subset\mathbb{R}^N$ with sides parallel to the coordinate axes, $|Q|$ --- its volume. 

Let $L_r(\mathbb{R}^N)$, $0<r\le\infty$, denote the Lebesgue space of all measurable functions $f$ on $\mathbb{R}^N$ quasi--normed by 
$\|f\|_{L_r(\mathbb{R}^N)}:=\bigl(\int_{\mathbb{R}^N}|f(x)|^r\,dx\bigr)^{{1}/{r}}$ with the usual modification in the case $r=\infty$.

\begin{definition}{\rm (\cite[Chapter~V]{S}) (i) A weight $w$ belongs to the \textit{Muckenhoupt class} $\mathscr{A}_p$, $1<p<\infty$, if \begin{equation}\label{Mup}
\mathscr{A}_p(w):=\sup_{Q\subset\mathbb{R}^N}\frac{w(Q)}{|Q|}
\biggl(\frac{1}{|Q|}\int_{Q} w^{1-p'}\biggr)^{\frac{p}{p'}}<\infty;\end{equation}
(ii) $w\in\mathscr{A}_1$ if 
 \begin{equation}\label{Mu1}
\mathscr{A}_1(w):=\sup_{Q\subset\mathbb{R}^N}\frac{w(Q)}{|Q|}\,\|1/w\|_{L_\infty(Q)}<\infty;\end{equation}
(iii) the \textit{Muckenhoupt class} $\mathscr{A}_\infty$ is given by $\bigcup_{p\ge 1}\mathscr{A}_p.$
}\end{definition} Suprema in \eqref{Mup} and \eqref{Mu1} are taken over all cubes $Q\subset\mathbb{R}^N$.
The set $\mathscr{A}_\infty$ is stable with respect to translation, dilation and multiplication {by} a positive scalar, cf. \cite[Chapter~V]{S}. Moreover,\\ (P1) if $w\in\mathscr{A}_p$, $1<p<\infty$, then $w^{-p'/p}\in\mathscr{A}_{p'}$;\\
(P2) ({\it doubling property in a general form}) if $w\in\mathscr{A}_p$ then there exist constants $c,u>0$ such that $$w(Q_t)\le c\, t^u\, w(Q)$$ for all cubes $Q$ and their concentric cubes $Q_t$ with side lengths $l(Q)$ and $l(Q_t)$ satisfying $l(Q_t)=t\cdot l(Q)$, $t\ge 1$;\\
(P3) if $1\le p_1<p_2\le\infty$ then $\mathscr{A}_{p_1}\subset \mathscr{A}_{p_2}$;\\(P4) if $w\in\mathscr{A}_p$, $p>1$, then there exists some number $r<p$ such that $w\in\mathscr{A}_r$.\\ We refer to \cite[Chapter~V]{S} and e.g. \cite[Lemma 1.3]{HSc}, \cite[Lemma 2.3]{HP} for details.

The property (P4) {emerges} in the following definition of the number \begin{equation*}\label{r_0}r_0:= \inf\{r\ge 1\colon w\in\mathscr{A}_r\}<\infty, \qquad w\in\mathscr{A}_\infty,\end{equation*} playing a role of an important technical parameter in decomposition theorems.

A number of examples of weights $w$ belonging to the class $\mathscr{A}_\infty$ can be {seen} in e.g. \cite[Examples 1.5]{HSc}, \cite[Remark 2.4, Example 2.7]{HP} and \cite[\S~2.1.2]{WM}. One of typical representatives of the Muckenhoupt weight class is the following:
$$ w(x)=|x|^\alpha\in\mathscr{A}_p, \quad\textrm{where}\quad \begin{cases} -N<\alpha<N(p-1), & p>1,\\ -N<\alpha\le 0, &p=1.\end{cases}$$ 
Alternative definitions, further properties and examples of Muckenhoupt weights can be found in \cite{S,WM} and \cite[Lemma 1.4]{HSc} (see also references given there).

\subsubsection{Local Muckenhoupt weight functions $\mathscr{A}_\infty^\loc$} Here we introduce a class of weight functions covering the sets of admissible, locally regular and weights of Muckenhoupt type $\mathscr{A}_\infty$.

\begin{definition}\label{defin}{\rm (\cite{R}) (i) A weight $w$ belongs to the class $\mathscr{A}_p^\loc$, $1<p<\infty$, of \textit{local Muckenhoupt weights} if \begin{equation}\label{Mup'}
\mathscr{A}_p^\loc(w):=\sup_{|Q|\le 1}\frac{w(Q)}{|Q|}
\biggl(\frac{1}{|Q|}\int_{Q} w^{1-p'}\biggr)^{\frac{p}{p'}}<\infty;\end{equation}\\
(ii) $w\in\mathscr{A}_1^\loc$ if \begin{equation}\label{Mu1'}
\mathscr{A}_1^\loc(w):=\sup_{|Q|\le 1}\frac{w(Q)}{|Q|}\,\|1/w\|_{L_\infty(Q)}<\infty;\end{equation}\\
(iii) we say that $w\in\mathscr{A}_\infty^\loc$ if $w\in \mathscr{A}_p^\loc$ for some $1\le p<\infty$, that is $\mathscr{A}_\infty^\loc:=\bigcup_{p\ge 1}\mathscr{A}_p^\loc.$
}\end{definition} 
Suprema in \eqref{Mup'} and \eqref{Mu1'} are taken over all cubes $Q\subset\mathbb{R}^N$ with $|Q|\le 1$. Similarly to $\mathscr{A}_\infty$, the class $\mathscr{A}_\infty^\loc$ is stable with respect to translation, dilation and multiplication {by} a positive scalar.
Besides,\\ (\textbf{P}1) if $w\in\mathscr{A}_p^\loc$, $1<p<\infty$, then $w^{-p'/p}\in\mathscr{A}_{p'}^\loc$;\\
(\textbf{P}2) if $w\in\mathscr{A}_p^\loc$, $1<p\le\infty$, then there exists a constant $c_{w,N}>0$ such that $$w(Q_t)\le \exp(c_{w,N}\, t) w(Q)$$ holds for all cubes $Q$ with $|Q|=1$ and for their concentric cubes $Q_t$ 
with side lengths $l(Q)$ and $l(Q_t)$ satisfying $l(Q_t)=t\cdot l(Q)$, $t\ge 1$;\\ (\textbf{P}3) if $1<p_1<p_2\le\infty$ then $\mathscr{A}_{p_1}^\loc\subset \mathscr{A}_{p_2}^\loc$.

Definition \ref{defin}(iii) implies that if $\mathscr{A}_{\infty}^\loc$ then $w\in\mathscr{A}_{p}^\loc$ for some $p<\infty$. In consequence, for $w\in\mathscr{A}_\infty^\loc$ one can define a positive number  \begin{equation}\label{r_w}\mathbf{r}_0:= \inf\{r\ge 1\colon w\in\mathscr{A}_r^\loc\}<\infty, \qquad w\in\mathscr{A}_\infty^\loc,\end{equation} similarly to $r_0$ for $w\in \mathscr{A}_{\infty}$. There is {\it the basic property} of $\mathscr{A}_p^\loc-$weights with respect to $\mathscr{A}_p-$weights \cite[Lemma 1.1]{R}: any $w\in\mathscr{A}_p^\loc$ can be extended beyond a given cube $Q$ so that the extended function $\bar{w}$ belongs to $\mathscr{A}_p$. \begin{lemma} {\rm (\cite[Lemma 1.1]{R})}
Let $1\le p<\infty$, $w\in\mathscr{A}_p^\loc$ and $Q$ be a unit cube, i.e., $|Q|=1$. Then there exists a $\bar{w}\in\mathscr{A}_p$ so that $\bar{w}=w$ on $Q$ and, with a positive constant $c$ independent of $Q$, the following estimate is satisfied: $$\mathscr{A}_p(\bar{w})\le c \mathscr{A}_p^\loc(w).$$
\end{lemma}

As an example of a weight, which is in $\mathscr{A}_\infty^\loc$ but is not in $\mathscr{A}_\infty$ and is not locally regular, one can take 
$$\mathscr{A}_p^\loc\ni w(x)=\begin{cases}|x|^\alpha, & |x|\le 1,\\
\exp(|x|-1), & |x|>1,\end{cases}\quad\textrm{where}\quad \begin{cases} -N<\alpha<N(p-1), & p>1,\\ -N<\alpha\le 0, &p=1.\end{cases}$$ 

\noindent More information and examples of local Muckenhoupt weights can be found in \cite{R,Ma,WM}.

\subsection{Function spaces} For $0<p<\infty$ and a weight $w$ on $\mathbb{R}^N$ we denote $L_p(\mathbb{R}^N,w)$ the weighted Lebesgue space quasi--normed by $\|f\|_{L_p(\mathbb{R}^N,w)}:=\|w^{1/p}f\|_{L_p(\mathbb{R}^N)}$ with usual modification if $p=\infty$.

For the definitions of the unweighted Besov $B_{pq}^s(\mathbb{R}^N)$ and Triebel--Lizorkin $F_{pq}^s(\mathbb{R}^N)$ spaces we refer to \cite{Tr1, Tr2}. 
Their weighted counterparts can be introduced in several ways \cite{Sch}. The space $\mathscr{S}'(\mathbb{R}^N)$ of tempered distributions suits well for the case of Muckenhoupt weights, (see \S~\ref{S}~or e.g. \cite{HTr2, HSc, Tr1}), but it is not sufficient when dealing with the local Muckenhoupt weight class (see \S~\ref{E}~for details). 
\subsubsection{Function spaces $B_{p,q}^s(\mathbb{R}^N,w)$ and $F_{p,q}^s(\mathbb{R}^N,w)$ with $w\in\mathscr{W}^N\cup\mathscr{A}_\infty$}\label{S} 
Let 
$\mathscr{S}(\mathbb{R}^N)$ be the Schwartz space of all complex--valued rapidly decreasing, infinitely differentiable functions on $\mathbb{R}^N$. By $\mathscr{S}'(\mathbb{R}^N)$ we denote its topological dual, the space of tempered distributions. If $\varphi\in\mathscr{S}(\mathbb{R}^N)$ then 
\begin{equation} \label{FourierS}
\widehat{\varphi}(\xi)=(F\varphi)(\xi)=(2\pi)^{-N/2}\int_{\mathbb{R}^N}\mathrm{e}^{-i\xi x}\varphi(x)\,dx, \qquad \xi\in\mathbb{R}^N,
\end{equation}
the Fourier transform of $\varphi$, in addition, $\widehat{\varphi}\in\mathscr{S}(\mathbb{R}^N)$. Symbol $F^{-1}\varphi$ stands for the inverse Fourier transform given by the right--hand side of \eqref{FourierS} with $i$ in place of $-i$. Both $F$ and $F^{-1}$ are extended to $\mathscr{S}'(\mathbb{R}^N)$ in the standard way.

Let $\varphi_0=\varphi\in\mathscr{S}(\mathbb{R}^N)$ be such that 
\begin{equation*}\mathrm{supp}\,\varphi\subset \{y\in\mathbb{R}^N\colon |y|<2\}\qquad \textrm{and}\quad \varphi(x)=1 \quad \textrm{ if }\quad |x|\le 1,
\end{equation*}
and for $d\in\mathbb{N}$ let $\varphi_d(x)=\varphi(2^{-d}x)-\varphi(2^{-d+1}x)$.  
Then $\{\varphi_d\}_{d=0}^\infty$ is a smooth dyadic resolution of unity. If $f\in\mathscr{S}'(\mathbb{R}^N)$ then the compact support of $\varphi_d\widehat{f}$ implies, by the Paley--Wiener--Schwartz theorem, that $F^{-1}(\varphi_d\widehat{f})$ is an entire analytic function on $\mathbb{R}^N$.

\begin{definition}\label{SDef2}{\rm
Let $0<p<\infty$, $0<q\le\infty$, $-\infty<s<+\infty$ and $\{\varphi_d\}_{d\in\mathbb{N}_0}$ be a smooth dyadic resolution of unity. Assume that $w\in\mathscr{W}^N\cup\mathscr{A}_\infty$.\\
(i) The weighted Besov space $B_{pq}^s(\mathbb{R}^N, w)$ is the collection of all distributions $f \in \mathscr{S}'(\mathbb{R}^N)$ with finite (quasi--)norm
\begin{equation}\label{BspqS}
  \|f\|_{B_{pq}^s(\mathbb{R}^N, w)}: =\Biggl(\sum_{d=0}^\infty 2^{dsq}\Bigl\| F^{-1}(\varphi_d\widehat{f})\Bigr\|_{L_p(\mathbb{R}^N,w)}^q\Biggr)^{\frac{1}{q}} 
\end{equation}
(with the usual modification if $q=\infty$).\\
 (ii) The weighted Triebel--Lizorkin space $F_{pq}^s(\mathbb{R}^N, w)$ consists of all $f \in \mathscr{S}'(\mathbb{R}^N)$ satisfying the condition
\begin{equation}\label{FspqS}
  \|f\|_{F_{pq}^s(\mathbb{R}^N, w)} =\Biggl\|\biggl(\sum_{d=0}^\infty 2^{dsq} \bigl|F^{-1}(\varphi_d\widehat{f})(\cdot)\bigr|^q\biggr)^{\frac{1}{q}} \Biggr\|_{L_p(\mathbb{R}^N,w)}<\infty
\end{equation}
(with the usual modification in the right hand side of \eqref{FspqS} for $q=\infty$).
}\end{definition}

Definitions of the above spaces are independent of the choice of $\varphi$, up to equivalence of quasi--norms. To simplify the notation we write $A^{s}_{pq}(\mathbb{R}^N,w)$ instead of $B_{pq}^{s}(\mathbb{R}^N,w)$ or $F_{pq}^{s}(\mathbb{R}^N,w)$.

If $w=1$ then we have the unweighted Besov $B_{pq}^s(\mathbb{R}^N)$ and Triebel--Lizorkin $F_{pq}^s(\mathbb{R}^N)$ spaces defined by \eqref{BspqS} and \eqref{FspqS} with $L_p(\mathbb{R}^N)$ instead of $L_p(\mathbb{R}^N,w)$. The theory and properties of $B_{pq}^s(\mathbb{R}^N, w)$ and $F_{pq}^s(\mathbb{R}^N, w)$ can be found in \cite{HTr1, HTr1',EdTr}.

\subsubsection{Function spaces $B_{pq}^{s,w}(\mathbb{R}^N)$ and $F_{pq}^{s,w}(\mathbb{R}^N)$ with $w\in\mathscr{A}_\infty^\loc$}\label{E} Let $\mathscr{D}(\mathbb{R}^N)$ denote the space of
all compactly supported $C^\infty(\mathbb{R}^N)$ functions equipped with the usual topology. To incorporate the $\mathscr{A}_\infty^\loc-$weights into the theory of weighted function spaces V.S. Rychkov exploited a class $\mathscr{S}'_e(\mathbb{R}^N)$ of distributions generalising $\mathscr{S}'(\mathbb{R}^N)$ (see \cite{R} and works \cite{Sch,Sch'} by Th. Schott where the class $\mathscr{S}_e(\mathbb{R}^N)$ was introduced). 
\begin{definition}{\rm 
The space
$\mathscr{S}_e(\mathbb{R}^N)$ consists of all
$C^\infty(\mathbb{R}^N)$ functions $\boldsymbol{\varphi}$ satisfying $$\boldsymbol{q}_\mathbf{N}(\boldsymbol{\varphi}):=\sup_{x\in\mathbb{R}^N}\mathrm{e}^{\mathbf{N}|x|}\sum_{|\gamma|\le \mathbf{N}}|D^\gamma\boldsymbol{\varphi}(x)|<\infty\qquad\textrm{for all}\quad\mathbf {N}\in\mathbb{N}_0.$$ }\end{definition} 
The set
$\mathscr{S}_e(\mathbb{R}^N)$ is equipped with the locally convex topology defined by the system of the semi--norms $\boldsymbol{q}_\mathbf{N}$. The following properties take place (see \cite{Sch}, \cite{R} and \cite[\S~3]{Ma}):\\ (i) $\mathscr{S}_e(\mathbb{R}^N)$ is a complete locally convex space;\\ (ii) $\mathscr{D}(\mathbb{R}^N)\hookrightarrow\mathscr{S}_e(\mathbb{R}^N) \hookrightarrow \mathscr{S}(\mathbb{R}^N)$;\\(iii) the space $\mathscr{D}(\mathbb{R}^N)$ is dense in $\mathscr{S}_e(\mathbb{R}^N)$; the space $\mathscr{S}_e(\mathbb{R}^N)$ is dense in $\mathscr{S}(\mathbb{R}^N)$;\\ (iv) if $w\in \mathscr{A}_\infty^\loc$ then $\mathscr{S}_e(\mathbb{R}^N)\hookrightarrow L_p(\mathbb{R}^N,w)$ for any $0<p<\infty$.

By $\mathscr{S}'_e(\mathbb{R}^N)$ we denote the strong dual space of  $\mathscr{S}_e(\mathbb{R}^N)$. The class $\mathscr{S}'_e(\mathbb{R}^N)$ can be identified with a subset of the collection $\mathscr{D}'(\mathbb{R}^N)$ of all distributions $f$ on $\mathscr{D}(\mathbb{R}^N)$ for which the estimate $$|\langle f,\boldsymbol{\varphi}\rangle|\le C\sup\Bigl\{\bigl|D^\gamma \boldsymbol{\varphi}(x)\bigr|\mathrm{e}^{\mathbf{N}|x|}\colon x\in\mathbb{R}^N,\,|\gamma|\le \mathbf{N}\Bigr\}$$ holds for all $\boldsymbol{\varphi}\in\mathscr{D}(\mathbb{R}^N)$ with some constants $C$ and $\mathbf{N}$ depending on $f$. Such a distribution $f$ can be extended to a continuous functional on $\mathscr{S}_e(\mathbb{R}^N)$.

Contrary to the situation with $\mathscr{S}(\mathbb{R}^N)$ and $\mathscr{S}'(\mathbb{R}^N)$, the convolution transform, instead of the Fourier one, is justifiably used when dealing with $\mathscr{S}_e(\mathbb{R}^N)$ and $\mathscr{S}'_e(\mathbb{R}^N)$. Indeed, for any functions $\boldsymbol{\varphi},\boldsymbol{\psi}\in\mathscr{S}_e(\mathbb{R}^N)$ their convolution $$\boldsymbol{\varphi}\ast\boldsymbol{\psi}(x)=:\int_{\mathbb{R}^N}\boldsymbol{\varphi}(x-y)\boldsymbol{\psi}(y)\,dy,\quad x\in\mathbb{R}^N,$$
belongs to $\mathscr{S}_e(\mathbb{R}^N)$. Besides, for $f\in\mathscr{S}'_e(\mathbb{R}^N)$ and $\boldsymbol{\varphi}\in\mathscr{S}_e(\mathbb{R}^N)$ the related convolution $$f\ast\boldsymbol{\varphi}(x)=:\langle f(\cdot),\boldsymbol{\varphi}(x-\cdot)\rangle,\quad x\in\mathbb{R}^N,$$
is a function from $C^\infty(\mathbb{R}^N)$ of at most exponential growth.
 
To be able to define the required spaces we take a function $\boldsymbol{\varphi}_0\in\mathscr{D}(\mathbb{R}^N)$ such that \begin{equation}\label{Iphi1}\int_{\mathbb{R}^N}\boldsymbol{\varphi}_0(x)\,dx\not=0,\end{equation} put 
\begin{equation}\label{Iphidef}\boldsymbol{\varphi}(x)=\boldsymbol{\varphi}_0(x)-2^{-N}\boldsymbol{\varphi}_0(x/2)\end{equation} and let $\boldsymbol{\varphi}_d(x):=2^{(d-1)N}\boldsymbol{\varphi}(2^{d-1}x)$ for $d\in\mathbb{N}$. One can find $\boldsymbol{\varphi}_0$ such that \begin{equation}\label{Iphi0}\int_{\mathbb{R}^N}x^\gamma\boldsymbol{\varphi}(x)\,dx=0\end{equation} for any multi--index $\gamma\in\mathbb{N}^N_0$, $|\gamma|\le \Gamma$, where $\Gamma\in\mathbb{N}$ is fixed. We write $\Gamma=-1$ if \eqref{Iphi0} does not hold.

\begin{definition}\label{Def2}{\rm (\cite{R}) 
Let $0<p<\infty$, $0<q\le\infty$, $-\infty<s<+\infty$ and $w\in\mathscr{A}_\infty^\loc$. Let a function $\boldsymbol{\varphi}_0\in\mathscr{D}(\mathbb{R}^N)$ satisfy \eqref{Iphi1} and $\boldsymbol{\varphi}$ of the form \eqref{Iphidef} satisfy \eqref{Iphi0} with $|\gamma|\le \Gamma$, where $\Gamma\ge [s]$. We define\\
(i) weighted Besov space $B_{pq}^{s,w}(\mathbb{R}^N)$ to be the set of all $f \in \mathscr{S}'_e(\mathbb{R}^N)$ such that the quasi--norm
\begin{equation*}\label{Bspq}
  \|f\|_{B_{pq}^{s,w}(\mathbb{R}^N)}: =\biggl(\sum_{d=0}^\infty 2^{dsq}\bigl\| \boldsymbol{\varphi}_d\ast f\bigr\|_{L_p(\mathbb{R}^N,w)}^q\biggr)^{\frac{1}{q}} 
\end{equation*}
(with the usual modification if $q=\infty$) is finite; and\\
(ii) weighted Triebel--Lizorkin space $F_{pq}^{s,w}(\mathbb{R}^N)$ as the collection of all $f \in \mathscr{S}'_e(\mathbb{R}^N)$ with finite quasi--norm
\begin{equation*}\label{Fspq}
  \|f\|_{F_{pq}^{s,w}(\mathbb{R}^N)} =\biggl\|\biggl(\sum_{d=0}^\infty 2^{dsq} \bigl|\boldsymbol{\varphi}_d\ast f\bigr|^q\biggr)^{\frac{1}{q}} \biggr\|_{L_p(\mathbb{R}^N,w)}
\end{equation*}
admitting the standard modification for $q=\infty$.
}
\end{definition}

To simplify the notation we write $A^{s,w}_{pq}(\mathbb{R}^N)$ instead of $B_{pq}^{s,w}(\mathbb{R}^N)$ or $F_{pq}^{s,w}(\mathbb{R}^N)$.
The definitions of the above spaces are independent of the choice of $\boldsymbol{\varphi}_0$, up to equivalence of quasi--norms. 
Definition \ref{Def2} covers Definition \ref{SDef2} of $B_{pq}^s(\mathbb{R}^N, w)$ and $F_{pq}^s(\mathbb{R}^N, w)$ with Muckenhoupt and admissible weights $w$, and Besov and Triebel--Lizorkin spaces with locally regular weights \cite[\S~2.2 and Remark 2.23]{R}.

The spaces $A^{s,w}_{pq}(\mathbb{R}^N)$ have properties similar to the unweighted $B^{s}_{pq}(\mathbb{R}^N)$ and $F^{s}_{pq}(\mathbb{R}^N)$ and spaces $A^{s}_{pq}(\mathbb{R}^N,w)$ with $w\in\mathscr{W}^N\cup\mathscr{A}_\infty$ (see \cite{IS,Ma,R,WBan,WM} for details).

\section{Spline wavelet systems with localisation properties}\label{SWS}

Consider a class of orthonormal spline wavelet systems of Battle--Lemari\'{e} type, which can be obtained by orthogonalisation process of (basic) $B-$splines. Detailed constructions of wavelet systems of such a type were established in \cite{JMAA, RMC} inspired by the idea taken from \cite{NS,NS1} (see also \cite{NPS}). To remind them briefly in \S~\ref{dusya} let us recall some fundamentals. 

Put $B_0=\chi_{[0,1)}$ and define \begin{equation}\label{Bndef}
B_n(x):=(B_{n-1}\ast B_0)(x)=\int_0^1 B_{n-1}(x-t)\,dt, \quad n\in\mathbb{N}.
\end{equation} Each $B-$spline $B_n$ of order $n$ is continuous and $n-$times a.e. differentiable function on $\mathbb{R}$ {with ${\rm supp}\,B_n=[0,n+1]$. In addition, $B_n(x)>0$ for all $x\in(0,n+1)$ and the restriction of $B_n$ to each $[m,m+1]$,} $m=0,\ldots,n$, is a polynomial of degree $n$. The function $B_n(x)$ is symmetric about $x=(n+1)/2$ (see e.g. \cite{Chui}).

$B$--splines satisfy the following differentiation property
\begin{equation}\label{diff}
B'_n(\cdot)=B_{n-1}(\cdot)-B_{n-1}(\cdot-1)\quad\textrm{almost everywhere (a.e.) on }\mathbb{R},\end{equation} 
or, in its generalised form,
\begin{equation}\label{diff'}
B^{(k)}_{n}(\cdot)=\sum_{\nu=0}^k \frac{(-1)^\nu k!}{\nu!(k-\nu)!}B_{n-k}(\cdot-\nu),\quad\quad n\ge k\in\mathbb{N}.\end{equation}

Orthogonalisation of $B$--splines of the form \eqref{Bndef} results in functions $\phi^{BL}_n\in L_2(\mathbb{R})$, named after G. Battle \cite{B,B1} and P.G. Lemari\'{e}--Rieusset \cite{L}. A function $\phi^{BL}_n$ satisfying
\begin{equation*}\label{BLphi}\hat{\phi}^{BL}_n(\omega)=\hat{B}_n(\omega)\left(\sum_{m\in\mathbb{Z}}\left|\hat{B}_n(\omega+2\pi m)\right|^2\right)^{-1/2}\end{equation*} is called the Battle--Lemari\'{e} scaling function of order $n$. The $n-$th order Battle--Lemari\'{e} wavelet related to $\phi^{BL}_n$ is a function $\psi^{BL}_n$ whose Fourier transform is
\begin{equation*}\label{BLpsi}\hat{\psi}^{BL}_n(\omega)=-\mathrm{e}^{-i\omega/2}\frac{\overline{\hat{\phi}^{BL}_n(\omega+2\pi)}}
{\overline{\hat{\phi}^{BL}_n(\omega/2+\pi)}}{\hat{\phi}^{BL}_n(\omega/2)}.\end{equation*}

\label{MRA}
{For any fixed} $\boldsymbol{k}\in\mathbb{Z}$ the system $\bigl\{{\phi}^{BL}_{n,\boldsymbol{k}}(\cdot-\tau):={\phi}^{BL}_n(\cdot-\boldsymbol{k}-\tau)\colon \tau\in\mathbb{Z}\bigr\}$ generates multiresolution analysis $\mathrm{MRA}_{{\phi}^{BL}_{n,\boldsymbol{k}}}$ of $L_2(\mathbb{R})$, and $\mathrm{MRA}_{{\phi}^{BL}_{n,\boldsymbol{k}}}=\mathrm{MRA}_{{\phi}^{BL}_n}$. Moreover, $\mathrm{MRA}_{{\phi}^{BL}_{n}}$ $=\mathrm{MRA}_{{B}_n}=\mathrm{MRA}_{{B}_{n,\boldsymbol{k}}}$. 
In more detail, let $\mathscr{V}_d$, $d\in\mathbb{Z}$, denote the $L_2(\mathbb{R})$ closure of the linear span of the system $\bigl\{B_{n}(2^d\cdot-\tau)\colon \tau\in\mathbb{Z}\bigr\}$. The spline spaces $\mathscr{V}_d$, $d\in\mathbb{Z}$, are generated by $B_n$ and constitute $\mathrm{MRA}_{B_{n}}$ of $L_2(\mathbb{R})$ in the sense that \begin{itemize}
\item[{\rm (i)}] $\ldots\subset \mathscr{V}_{-1}\subset \mathscr{V}_0\subset \mathscr{V}_1\subset\ldots$;
\item[{\rm (ii)}] $\mathrm{clos}_{L_2(\mathbb{R})}\Bigl(\bigcup_{d\in\mathbb{Z}} \mathscr{V}_d\Bigr)=L_2(\mathbb{R})$;
\item[{\rm (iii)}] $\bigcap_{d\in\mathbb{Z}} \mathscr{V}_d=\{0\}$;
\item[{\rm (iv)}] for each $d$ the system $\bigl\{B_{n}(2^d\cdot-\tau)\colon\tau\in\mathbb{Z}\bigr\}$ is an unconditional basis of $\mathscr{V}_d$.\end{itemize} Observe that $\bigl\{B_{n}(2^d\cdot-\tau)\colon\tau\in\mathbb{Z}\bigr\}$ is not orthonormal.
Further, there are the orthogonal complementary subspaces $\ldots, \mathscr{W}_{-1},\mathscr{W}_0,\mathscr{W}_1,\ldots$ such that 
\begin{itemize}
\item[{\rm (v)}] $\mathscr{V}_{d+1}=\mathscr{V}_d\oplus \mathscr{W}_d$ for all $d\in\mathbb{Z}$,
where $\oplus$ stands for $\mathscr{V}_d\perp \mathscr{W}_d$ and $\mathscr{V}_{d+1}=\mathscr{V}_d+\mathscr{W}_d$. 
\end{itemize} Wavelet subspaces $\mathscr{W}_d$, $d\in\mathbb{Z}$, related to the spline $B_n$, are also generated by some basis functions ({\it wavelets}) in the same manner as the spline spaces $\mathscr{V}_d$, $d\in\mathbb{Z}$, are generated by the scaling function $B_n$.

The scaling function $B_n$ and one of its associated wavelets form a wavelet system generating a Riesz basis of $L_2(\mathbb{R})$. Unfortunately, not all the elements of such a basis are orthogonal to each other \cite{Chui,W}.  In contrast, Battle--Lemari\'{e} systems $\{{\phi}^{BL}_n,{\psi}^{BL}_n\}$ constitute orthonormal bases of $L_2(\mathbb{R})$ within $\mathrm{MRA}_{{B}_n}$.

\subsection{Battle--Lemari\'{e} wavelet systems of natural orders} \label{dusya}
Constructions of the required spline wavelet systems $\{\phi^{BL}_n, \psi^{BL}_n\}$, in relativity explicit forms, were established in \cite[\S\,2.2]{JMAA} and \cite[\S\,3]{RMC}.

Fix $n\in\mathbb{N}$. For each $j=1,\ldots, n$ we define $r_{j}(n):=(2\alpha_j(n)-1)-2\sqrt{\alpha_j(n)(\alpha_j(n)-1)}$ with some particular $\alpha_{j}(n)>1$. Then $r_{j}(n)\in(0,1)$ for all $j=1,\ldots,n$.  Put $\beta_n:=2^n\sqrt{\alpha_1(n)\,r_1(n)\ldots\alpha_n(n)\,r_n(n)}$ and define the $n-$th order Battle--Lemari\'{e} scaling function $\phi_{n,\boldsymbol{k}}$ via its Fourier transform as follows:
\begin{equation}\label{phi_nn}
\hat{\phi}_{n,\boldsymbol{k}}(\omega):={\beta_n\,\hat{B}_{n,\boldsymbol{k}}(\omega)}{\mathbf{A}_n^{-1}(\omega)}=m_{n,\boldsymbol{k}}(\omega/2)\,\hat{\phi}_{n,\boldsymbol{k}}(\omega/2),\end{equation}
where $\mathbf{A}_n(\omega):=\bigl(1+\mathrm{e}^{i\omega}r_1(n)\bigr)\ldots\bigl
(1+\mathrm{e}^{i\omega}r_n(n)\bigr)$, $$m_{n,\boldsymbol{k}}(\omega):=
\mathrm{e}^{-i(n+1)\omega/2}(\cos(\omega/2))^{n+1}\frac{\mathbf{A}_n(\omega)}{\mathbf{A}_n(2\omega)}\,\mathrm{e}^{-i\omega\boldsymbol{k}}.$$ Here the parameter $\boldsymbol{k}\subset\mathbb{Z}$ is fixed. It allows to start with ${B}_{n,\boldsymbol{k}}$ centred at $(n+1)/2+\boldsymbol{k}$.
Since 
$$\sum_{m\in\mathbb{Z}}\Bigl|\hat{\phi}_{n,\boldsymbol{k}}(\omega+2\pi m)\Bigr|^2=
\beta_n^2\bigl|\mathbf{A}_n(\omega)\bigr|^{-2}\sum_{m\in\mathbb{Z}}\Bigl|\hat{B}_{n,\boldsymbol{k}}(\omega+2\pi m)\Bigr|^2$$ and (see \cite[\S\,2]{JMAA}) $$\mathbb{P}_{n,\boldsymbol{k}}(\omega):=\sum_{m\in\mathbb{Z}}\Bigl|\hat{B}_{n,\boldsymbol{k}}(\omega+2\pi m)\Bigr|^2=\sum_{m\in\mathbb{Z}}\Bigl|\hat{B}_n(\omega+2\pi m)\Bigr|^2=\frac{1}{\beta_n^2}\bigl|\mathbf{A}_n(\omega)\bigr|^{2},$$ then
$$\sum_{m\in\mathbb{Z}}\Bigl|\hat{\phi}_{n,\boldsymbol{k}}(\omega+2\pi m)\Bigr|^2=
1.$$ Therefore, {for fixed} $\boldsymbol{k}\in\mathbb{Z}$ the system $\{\phi_{n,\boldsymbol{k}}(\cdot-\tau)\colon \tau\in\mathbb{Z}\}$ forms an orthonormal basis in $\mathscr{V}_0$ of $\mathrm{MRA}_{{B}_n}$.

Since \begin{equation}\label{ryadok}({\mathrm{e}^{\pm i\omega}r+1})^{-1}=\sum_{l=0}^\infty \left(-r\,\mathrm{e}^{\pm i\omega}\right)^l,\qquad0<r<1,\end{equation} it follows from \eqref{phi_nn}, that
\begin{equation*}\label{ex2}\hat{\phi}_{n,\boldsymbol{k}}(\omega)=\frac{\beta_n\ \hat{B}_{n,\boldsymbol{k}}(\omega)}{\prod_{j=1}^n \bigl(1+\mathrm{e}^{i\omega}r_j(n)\bigr)}=
\beta_n \prod_{j=1}^n\sum_{l_j=0}^\infty \bigl(-r_j(n)\,\mathrm{e}^{i\omega}\bigr)^{l_j}\,\hat{B}_{n,\boldsymbol{k}}(\omega),\end{equation*} that is, 
\begin{equation*}
{\phi}_{n,\boldsymbol{k}}(x)=\beta_n\sum_{l_1\ge 0}\bigl(-r_1(n)\bigr)^{l_1}\ldots\sum_{l_n\ge 0}\bigl(-r_n(n)\bigr)^{l_n}B_{n,\boldsymbol{k}}\bigl(x+l_1+\ldots+ l_n\bigr).\end{equation*}
In particular, \begin{equation*}
\phi_{1,\boldsymbol{k}}(x)=\beta_1\sum_{l\ge 0}\bigl(-r_1(1)\bigr)^l B_{1,\boldsymbol{k}}(x+l),
\qquad
B_1(x)=\begin{cases} x, & \quad\textrm{if}\quad 0\le x<1,\\ 2-x, & \quad\textrm{if}\quad 1\le x<2,\\ 0, & \quad\textrm{otherwise};\end{cases}\end{equation*} 
\begin{multline*}
\phi_{2,\boldsymbol{k}}(x)=\beta_2
\sum_{l_1=0}^\infty \bigl(-r_1(2)\bigr)^{l_1}
\sum_{l_2=0}^\infty \bigl(-r_2(2)\bigr)^{l_2}B_{2,\boldsymbol{k}}(x+ l_1+ l_2)
,\\
B_2(x)=\begin{cases}
x^2/2, & 0\le x<1,\\
-x^2+3x-{3}/{2}, & 1\le x<2,\\
x^2/2-3x+{9}/{2}, & 2\le x<3,\\ 0 & \textrm{otherwise}.
\end{cases}\end{multline*}

The Fourier transform of a wavelet function ${\psi}_{n,\boldsymbol{k},\boldsymbol{s}}$ related to the $\phi_{n,\boldsymbol{k}}$ must satisfy the condition
\begin{equation*}
\hat{\psi}_{n,\boldsymbol{k},\boldsymbol{s}}(\omega)={M}_{n,\boldsymbol{k},\boldsymbol{s}}(\omega/2)\hat{\phi}_{n,\boldsymbol{k}}(\omega/2)\qquad(\boldsymbol{s}\in\mathbb{Z}),\end{equation*} where we can put (see \cite[\S~2]{JMAA} for details)
\begin{equation*}
{M}_{n,\boldsymbol{k},\boldsymbol{s}}(\omega):=\mathrm{e}^{-i\omega}\,\overline{m_{n,\boldsymbol{k}}(w+\pi)}\,\mathrm{e}^{-2i\omega\boldsymbol{s}}.\end{equation*} Similarly to the situation with $\boldsymbol{k}$, the parameter $\boldsymbol{s}\in\mathbb{Z}$ is also fixed here. It makes possible the positioning  ${\psi}_{n,\boldsymbol{k},\boldsymbol{s}}$ at a particular point on $\mathbb{R}$ (see e.g. \cite[p. 25]{RMC}). 

Denote $\mathscr{A}_n(\omega):=\overline{\mathbf{A}_n(\omega+\pi)}=\bigl(1-\mathrm{e}^{-i\omega}r_1(n)\bigr)\ldots\bigl
(1-\mathrm{e}^{-i\omega}r_n(n)\bigr)$. 
Since
$$m_{n,\boldsymbol{k}}(\omega+\pi)=
\mathrm{e}^{-i(n+1)(\omega+\pi)/2}(\mathrm{e}^{-i\pi}\sin(\omega/2))^{n+1}
\frac{\mathscr{A}_n(-\omega)}{\mathbf{A}_n(2\omega)}\,\mathrm{e}^{-i(\omega+\pi)\boldsymbol{k}},$$  
then
\begin{align*}
{M}_{n,\boldsymbol{k},\boldsymbol{s}}(\omega)=&\mathrm{e}^{-i\omega}\mathrm{e}^{i(n+1)(\omega+\pi)/2}(\mathrm{e}^{i\pi}\sin(\omega/2))^{n+1}
\frac{\mathscr{A}_n(\omega)}{\mathbf{A}_n(-2\omega)}\,\mathrm{e}^{i(\omega+\pi)\boldsymbol{k}}\,\mathrm{e}^{-2i\omega\boldsymbol{s}}\nonumber\\
=&\mathrm{e}^{i\pi(n+1+\boldsymbol{k})}\mathrm{e}^{-i\omega}\left(\frac{\mathrm{e}^{i\omega}-1}{2}\right)^{n+1}
\frac{\mathscr{A}_n(\omega)}{\mathbf{A}_n(-2\omega)}\,\mathrm{e}^{i\omega\boldsymbol{k}}\,\mathrm{e}^{-2i\omega\boldsymbol{s}}.\end{align*} Therefore,
\begin{equation}\label{Npsi_n}\hat{\psi}_{n,\boldsymbol{k},\boldsymbol{s}}(\omega)=\frac{\beta_n\,\mathrm{e}^{-i\omega\boldsymbol{s}}\,\mathrm{e}^{-i\omega/2}}{2^{n+1}\,\mathrm{e}^{i\pi(n+1+\boldsymbol{k})}}\,
\frac{\mathscr{A}_n(\omega/2)\,\bigl({\mathrm{e}^{i\omega/2}-1}\bigr)^{n+1}}{\mathbf{A}_n(-\omega)\mathbf{A}_n(\omega/2)}\,\hat{B}_{n}(\omega/2).\end{equation}
Observe that, the pre--image of 
$$\mathrm{e}^{-i\omega/2}\bigl(\mathrm{e}^{i\omega/2}-1\bigr)^{n+1}=\sum_{k=0}^{n+1}\frac{(-1)^k(n+1)!}{k!(n+1-k)!}\mathrm{e}^{(n-k)i\omega/2}\,\hat{B}_{n}(\omega/2)$$   
is exactly
$$
\sum_{k=0}^{n+1}\frac{(-1)^k(n+1)!}{k!(n+1-k)!}\ {B}_n(2\cdot +(n-k)),$$ which is equal to the {$(n+1)-$st} order derivative of $2^{-n-1}B_{2n+1}(2x+n)$.

Multiplying the numerator and denominator in \eqref{Npsi_n} by $\mathscr{A}_n(-\omega/2)$ we obtain
\begin{equation}\label{Npsi_nL}\hat{\psi}_{n,\boldsymbol{k},\boldsymbol{s}}(\omega)=\frac{\beta_n\,\mathrm{e}^{-i\omega\boldsymbol{s}}\,\mathrm{e}^{-i\omega/2}}{2^{n+1}\,\mathrm{e}^{i\pi(n+1+\boldsymbol{k})}}\,
\frac{\bigl|\mathscr{A}_n(\pm\omega/2)\bigr|^2\,\bigl({\mathrm{e}^{i\omega/2}-1}\bigr)^{n+1}}{\mathbf{A}_n(-\omega)\mathcal{A}_n(\omega)}\,\hat{B}_{n}(\omega/2)\end{equation}
with $\mathcal{A}_n(\omega):=\mathbf{A}_n(\omega/2)\mathscr{A}_n(-\omega/2)=\bigl(1-\mathrm{e}^{i\omega}r_1^2(n)\bigr)\ldots\bigl
(1-\mathrm{e}^{i\omega}r_n^2(n)\bigr)$. 

Denote $\rho_j(n)=r_j(n)+1/r_j(n)$, $j=1,\ldots,n$. Since
\begin{equation*}\label{modul}
[1-\mathrm{e}^{- i\omega/2}r]\,[1-\mathrm{e}^{i\omega/2}r]=|1-\mathrm{e}^{\pm i\omega/2}r|^2=r\Bigl[\bigl(r+1/r\bigr)-\bigl(\mathrm{e}^{i\omega/2}+\mathrm{e}^{-i\omega/2}\bigr)\Bigr]\qquad(r>0),
\end{equation*} then
\begin{equation*}
\bigl|\mathscr{A}_n(\pm\omega/2)\bigr|^2=\bigl[r_1(n)\ldots r_n(n)\bigr]\sum_{j=0}^n(-1)^j\lambda_j(n)\cos(j\omega/2),\end{equation*} where $\lambda_n(n)=2$, $0<\lambda_j(n)=\lambda_j(\rho_1(n),\ldots,\rho_n(n))$ for $j\not=n$ and $\lambda_j(n)$ is even for all $j\not=0$.
From here and \eqref{Npsi_nL}, we obtain on the strength of \eqref{ryadok}:
\begin{multline}\label{Npsihat}\hat{\psi}_{n,\boldsymbol{k},\boldsymbol{s}}(w)=\frac{\bigl[r_1(n)\ldots r_n(n)\bigr]\beta_n}{2^{n+1}\cdot(-1)^{n+1+\boldsymbol{k}}}\biggl\{\sum_{j=0}^n\frac{\lambda_j(n)}{2(-1)^j}\bigl[\mathrm{e}^{ji\omega/2}+
\mathrm{e}^{-ji\omega/2}\bigr]\biggr\}\sum_{k=0}^{n+1}\frac{(-1)^k(n+1)!}{k!(n+1-k)!}\mathrm{e}^{(n-k)i\omega/2}\\
\times \prod_{j=1}^n\sum_{m_j=0}^\infty \bigl(-r_j(n)\,\mathrm{e}^{- i\omega}\bigr)^{m_j}\sum_{l_j=0}^\infty \bigl(r_j^2(n)\,\mathrm{e}^{ i\omega}\bigr)^{l_j}\ \hat{B}_{n}(\omega/2)\ \mathrm{e}^{ -i\omega\boldsymbol{s}}.\end{multline} In particular, with $\gamma_{n,\boldsymbol{k}}:=
[r_1(n)\ldots r_n(n)]\beta_n2^{-n}(-1)^{n+1+\boldsymbol{k}}$ and
$B_{n,\boldsymbol{s}}(2x):=B_{n}(2(x-\boldsymbol{s}))$:
\begin{multline*}\label{Stpsi}
{\psi}_{1,\boldsymbol{k},\boldsymbol{s}}(x)=
{\gamma_{1,\boldsymbol{k}}}\sum_{m\ge 0}\bigl(-r_1(1)\bigr)^{m}\sum_{l\ge 0}r_1^{2l}(1)\Bigl[-B_{1,\boldsymbol{s}}\bigl(2x-2m+2l+2\bigr)
+\bigl(\rho_1(1)+2\bigr)B_{1,\boldsymbol{s}}\bigl(2x-2m+2l+1\bigr)\\-2\bigl(\rho_1(1)+1\bigr)B_{1,\boldsymbol{s}}\bigl(2x-2m+2l\bigr)+\bigl(\rho_1(1)+2\bigr)B_{1,\boldsymbol{s}}\bigl(2x-2m+2l-1\bigr)-B_{1,\boldsymbol{s}}\bigl(2x-2m+2l-2\bigr)\Bigr]
\\=\frac{\gamma_{1,\boldsymbol{k}}}{4}
\sum_{m\ge 0}\bigl(-r_1(1)\bigr)^{m}\sum_{l\ge 0}r_1^{2l}(1)\Bigl[-B_{3,\boldsymbol{s}}^{''}\bigl(2x-2m+2l+2\bigr)\\
+\rho_1(1)B_{3,\boldsymbol{s}}^{''}\bigl(2x-2m+2l+1\bigr)-B_{3,\boldsymbol{s}}^{''}\bigl(2x-2m+2l\bigr)\Bigr];\end{multline*}
while, with $\lambda_0(2)=2+\rho_1(2)\rho_2(2)$ and $\lambda_1(2)=2(\rho_1(2)+\rho_2(2))$,  
\begin{multline*}
{\psi}_{2,\boldsymbol{k},\boldsymbol{s}}(x)=\gamma_{2,\boldsymbol{k}}\sum_{m_1\ge 0}\bigl(-r_1(2)\bigr)^{m_1}\sum_{m_2\ge 0}\bigl(-r_1(2)\bigr)^{m_2}\sum_{l_1\ge 0}r_1^{2l_1}(2)\sum_{l_2\ge 0}r_1^{2l_2}(2)\\\Bigl[B_{2,\boldsymbol{s}}\bigl(2x-2 m_1-2m_2+ 2l_1+2l_2+4\bigr)
-\Bigl(\frac{\lambda_1(2)}{2}+3\Bigr)B_{2,\boldsymbol{s}}\bigl(2x- 2m_1-2m_2+ 2l_1+2l_2+3\bigr)\\+\Bigl(\lambda_0(2)+\frac{3\lambda_1(2)}{2}+3\Bigr)B_{2,\boldsymbol{s}}\bigl(2x- 2m_1-2m_2+ 2l_1+2l_2+2\bigr)\\-\bigl(3\lambda_0(2)+2\lambda_1(2)+1\bigr)B_{2,\boldsymbol{s}}\bigl(2x- 2m_1-2m_2+ 2l_1+2l_2+1\bigr)\\+\bigl(3\lambda_0(2)+2\lambda_1(2)+1\bigr)B_{2,\boldsymbol{s}}\bigl(2x- 2m_1-2m_2+ 2l_1+2l_2\bigr)\\-\Bigl(\lambda_0(2)+\frac{3\lambda_1(2)}{2}+3\Bigr)B_{2,\boldsymbol{s}}\bigl(2x- 2m_1-2m_2+ 2l_1+2l_2-1\bigr)\\
+\Bigl(\frac{\lambda_1(2)}{2}+3\Bigr)B_{2,\boldsymbol{s}}\bigl(2x- 2m_1-2m_2+ 2l_1+2l_2-2\bigr)
-B_{2,\boldsymbol{s}}\bigl(2x- 2m_1-2m_2+ 2l_1+2l_2-3\bigr)\Bigr]
\\
=\frac{
\gamma_{2,\boldsymbol{k}}}{8}\sum_{m_1\ge 0}\bigl(-r_1(2)\bigr)^{m_1}\sum_{m_2\ge 0}\bigl(-r_1(2)\bigr)^{m_2}\sum_{l_1\ge 0}r_1^{2l_1}(2)\sum_{l_2\ge 0}r_1^{2l_2}(2)\\\Bigl[B_{5,\boldsymbol{s}}^{(3)}\bigl(2x-2m_1-2m_2+2l_1+2l_2+4\bigr)
-\frac{\lambda_1(2)}{2}B_{5,\boldsymbol{s}}^{(3)}\bigl(2x-2m_1-2m_2+2l_1+2l_2+3\bigr)\\+\lambda_0(2)B_{5,\boldsymbol{s}}^{(3)}\bigl(2x-2m_1-2m_2+2l_1+2l_2+2\bigr)-\frac{\lambda_1}{2}B_{5,\boldsymbol{s}}^{(3)}\bigl(2x-2m_1-2m_2+2l_1+2l_2+1\bigr)\\+B_{5,\boldsymbol{s}}^{(3)}\bigl(2x-2m_1-2m_2+2l_1+2l_2\bigr)\Bigr].\end{multline*}

By construction, orthonormal wavelet systems $\{
{\phi}_{n,\boldsymbol{k}}, {\psi}_{n,\boldsymbol{k},\boldsymbol{s}}\}$ are from $\mathrm{MRA}_{B_n}$  
for any $\boldsymbol{k},\boldsymbol{s}\in\mathbb{Z}$. 
Substitution $x=\tilde{x}+1/2$ into the definition of $B_n$ leads to  another type of Battle--Lemari\'{e} wavelet systems of natural orders $\{
\tilde{\phi}_{n,\boldsymbol{k}}, \tilde{\psi}_{n,\boldsymbol{k},\boldsymbol{s}}\}$, which are shifted with respect to $\{
{\phi}_{n,\boldsymbol{k}}, {\psi}_{n,\boldsymbol{k},\boldsymbol{s}}\}$ in $1/2$ to the left. These are from $\mathrm{MRA}_{\tilde{B}_n}$ generated by the shifted $B$--spline $\tilde{B}_n(x):=B_n(x+1/2)$. 
Basic properties of Battle--Lemari\'{e} wavelet systems are described in \cite[Proposition 3.1]{RMC}. These systems can be chosen to be $k$--smooth functions if $n\ge k+1$ having exponential decay with decreasing rate as $n$ increases\cite[\S~5.4]{Dau}.

\subsection{Localisation of elements of Battle--Lemari\'{e} wavelet systems}\label{NN} For $n\in\mathbb{N}$ the both ${\phi}_{n,\boldsymbol{k}}$ and ${\psi}_{n,\boldsymbol{k},\boldsymbol{s}}$ have unbounded supports on $\mathbb{R}$ (see \S~\ref{dusya}). Several the so called algorithms for localising $\{
{\phi}_{n,\boldsymbol{k}}, {\psi}_{n,\boldsymbol{k},\boldsymbol{s}}\}$ were established in \cite[\S~3]{JMAA} and \cite[\S~3.2.2]{RMC}. 
Here we introduce, in some sense, their mixed version. 

We have (see \eqref{phi_nn} and \eqref{Npsi_nL} with \eqref{Npsihat}):
\begin{equation}\label{phi_nn'}
\hat{\phi}_{n,\boldsymbol{k}}(\omega)=\beta_n\,\frac{\mathrm{e}^{-i\omega\boldsymbol{k}}\,\hat{B}_{n}(\omega)}{\mathbf{A}_n(\omega)},\end{equation}
\begin{equation*}\hat{\psi}_{n,\boldsymbol{k},\boldsymbol{s}}(\omega)=\beta_{n}\,(-1)^{n+1+\boldsymbol{k}}\,2^{-n-1}\,
\frac{\bigl|\mathscr{A}_n(\pm\omega/2)\bigr|^2\displaystyle\sum_{k=0}^{n+1}\frac{(-1)^k(n+1)!}{k!(n+1-k)!}\,\mathrm{e}^{(n-k)i\omega/2}\,\mathrm{e}^{-i\omega\boldsymbol{s}}\,\hat{B}_{n}(\omega/2)}{\mathbf{A}_n(-\omega)\mathcal{A}_n(\omega)}.\end{equation*}

Localisation algorithms are based on getting rid of denominators in $\hat{\phi}_{n,\boldsymbol{k}}$ and $\hat{\psi}_{n,\boldsymbol{k},\boldsymbol{s}}$, which, in fact, cause spreading over $\mathbb{R}$ the compactly supported pre--images of $\hat{B}_{n,\boldsymbol{k}}$ (see \eqref{phi_nn'}) and $$\bigl|\mathscr{A}_n(\pm\omega/2)\bigr|^2\sum_{k=0}^{n+1}\frac{(-1)^k(n+1)!}{k!(n+1-k)!}\mathrm{e}^{(n-k)i\omega/2}\,\mathrm{e}^{-i\omega\boldsymbol{s}}\,\hat{B}_{n}(\omega/2)$$ standing in numerators of $\hat{\phi}_{n,\boldsymbol{k}}$ and $\hat{\psi}_{n,\boldsymbol{k},\boldsymbol{s}}$, respectively.

Fix some $m\in\mathbb{N}$. One of localised versions of ${\phi}_{n,\boldsymbol{k}}$ can be represented by a function ${\mathbf{\Phi}}_{n,\boldsymbol{k}}$ \cite[\S~3]{JMAA} such that
\begin{equation}\label{phi_fint}
\hat{\mathbf{\Phi}}_{n,\boldsymbol{k}}(\omega)=\hat{\phi}_{n,\boldsymbol{k}}(\omega)\mathbf{A}_n(\omega)=\beta_n\,\hat{B}_{n,\boldsymbol{k}}(\omega).\end{equation}
Respectively to the form of $\mathbf{A}_n(\omega)$, the localised functions ${\mathbf{\Phi}}_{n,\boldsymbol{k}}$ are linear combinations of $n+1$ orthogonal to each other functions ${{\phi}}_{n,\boldsymbol{k}}$, with $\boldsymbol{k}$ replaced by $\boldsymbol{k}-n,\ldots,\boldsymbol{k}-1,\boldsymbol{k}$ and coefficients corresponding to the definition of $\mathbf{A}_n(\omega)$. The number of steps for performing ${\mathbf{\Phi}}_{n,\boldsymbol{k}}$ from integer translations of ${{\phi}}_{n,\boldsymbol{k}}$ is equal to $n$. According to the structure of $\mathbf{A}_n(\omega)$, at each step, we should sum an element obtained at the previous stage (or ${{\phi}}_{n,\boldsymbol{k}}$ in the beginning) with the same one, but shifted in 1 and multiplied by $r_l(n)$, $l=1,\ldots,n$.

As a localised analogue of ${\psi}_{n,\boldsymbol{k},\boldsymbol{s}}$, suitable for our purposes, we shall use a function $\mathbf{\Psi}_{n,m(\Bbbk);\boldsymbol{k},\boldsymbol{s}}$ satisfying the condition
\begin{multline}\label{Npsihat'}\hat{\mathbf{\Psi}}_{n,m(\Bbbk);\boldsymbol{k},\boldsymbol{s}}(\omega)=
\hat{\psi}_{n,\boldsymbol{k},\boldsymbol{s}}(\omega){\mathbf{A}_n(-\omega)\mathcal{A}_n(\omega)}\,\bigl|\mathcal{A}_{m}(\pm\omega)\bigr|^{2\Bbbk}\\=
\frac{\beta_{n}(-1)^{\boldsymbol{k}}}{(-2)^{n+1}}\,\bigl|\mathscr{A}_n(\pm\omega/2)\bigr|^2\,\bigl|\mathcal{A}_{m}(\pm\omega)\bigr|^{2\Bbbk}\sum_{k=0}^{n+1}\frac{(-1)^k(n+1)!}{k!(n+1-k)!}\mathrm{e}^{(n-k)i\omega/2}\,\mathrm{e}^{-i\omega\boldsymbol{s}}\,\hat{B}_{n}(\omega/2)\end{multline} with power $\Bbbk\in\{0,1\}$. If $\Bbbk=0$ then \begin{equation}\label{dop1}\hat{\mathbf{\Psi}}_{n,m(0);\boldsymbol{k},\boldsymbol{s}}=:\hat{\mathbf{\Psi}}_{n,\boldsymbol{k},\boldsymbol{s}}=\hat{\psi}_{n,\boldsymbol{k},\boldsymbol{s}}(\omega){\mathbf{A}_n(-\omega)\mathcal{A}_n(\omega)}\end{equation} (see \cite[\S~3.2.2]{RMC}). According to the form of the product $\mathbf{A}_n(-\omega)\mathcal{A}_n(\omega)$, the functions ${\mathbf{\Psi}}_{n,\boldsymbol{k},\boldsymbol{s}}$ are linear combinations of $2n+1$ mutually orthogonal elements ${\psi}_{n,\boldsymbol{k},\boldsymbol{s}}$, with $\boldsymbol{s}$ replaced by $\boldsymbol{s}-n,\ldots, \boldsymbol{s}, \ldots,\boldsymbol{s}+n$ and coefficients related to the definitions of $\mathbf{A}_n(-\omega)$ and $\mathcal{A}_n(\omega)$. Construction of ${\mathbf{\Psi}}_{n,\boldsymbol{k},\boldsymbol{s}}$ from 
${\psi}_{n,\boldsymbol{k},\boldsymbol{s}}$ requires $2n$ steps, and is fully dependent of the structure of $\mathbf{A}_n(-\omega)\mathcal{A}_n(\omega)$. 
Respectively, comparing the first line in \eqref{Npsihat'} for $\Bbbk=1$ with \eqref{dop1}, one can observe that ${\mathbf{\Psi}}_{n,m(1);\boldsymbol{k},\boldsymbol{s}}$ are linear combinations of ${\mathbf{\Psi}}_{n,\boldsymbol{k},\boldsymbol{s}}$ with $\boldsymbol{s}$ replaced by $\boldsymbol{s}-m,\ldots, \boldsymbol{s}, \ldots\boldsymbol{s}+m$ and coefficients related to the definition of $\mathcal{A}_m(-\omega)\mathcal{A}_m(\omega)$. Therefore, the total number of steps for establishing ${\mathbf{\Psi}}_{n,m(1);\boldsymbol{k},\boldsymbol{s}}$ from orthogonal to each other integer shifts of ${\psi}_{n,\boldsymbol{k},\boldsymbol{s}}$ equals to $2n+2m$.

Property \eqref{diff} (or \eqref{diff'}) provides possibilities to construct from ${\mathbf{\Phi}}_{n,\boldsymbol{k}}$ proper variants of functions such that their (anti)derivatives of order $|n-m|$ coincide with ${\mathbf{\Phi}}_{m,\boldsymbol{k}}$ for $m\not= n$. Analogously, one can use ${\mathbf{\Psi}}_{n,m(1);\boldsymbol{k},\boldsymbol{s}}$ to built functions having $|n-m|-$th order (anti)derivatives equal to some linear combinations of shifts of ${\mathbf{\Psi}}_{{m},\boldsymbol{k},\boldsymbol{s}}$ and ${\mathbf{\Psi}}_{{m},\boldsymbol{k},\boldsymbol{s}+1/2}$ with $m\not= n$.
This fact can be useful for the study of particular classes of operators in weighted function spaces, when one needs to transfer from a basis of smoothness $n$ to basis elements of some other smoothness, say $m$.  

On the strength of \eqref{phi_fint} and \eqref{Npsihat'}, the ${\mathbf{\Phi}}_{n,\boldsymbol{k}}$ and $\mathbf{\Psi}_{n,m(\Bbbk);\boldsymbol{k},\boldsymbol{s}}$ are compactly supported.
It also holds: \begin{equation}\label{Fu0}\bigl|\hat{\mathbf{\Phi}}_{n,\boldsymbol{k}}(\omega)\bigr|
=\bigl|\hat{\phi}_{n,\boldsymbol{k}}(\omega){\mathbf{A}_n(\omega)}\bigr|
=\beta_n\bigl|\hat{B}_{n;\boldsymbol{k}}(\omega)\bigr|,\end{equation}\begin{multline}\label{Fu}
\bigl|\hat{\mathbf{\Psi}}_{n,m(\Bbbk);\boldsymbol{k},\boldsymbol{s}}(\omega)\bigr|
=\bigl|\mathcal{A}_{m}(\pm\omega)\bigr|^{2\Bbbk}\bigl|\hat{\psi}_{n,\boldsymbol{k},\boldsymbol{s}}(\omega){\mathbf{A}_n(\omega)}\mathcal{A}_n(\omega)\bigr|\\=\frac{|{\gamma}_{n,\boldsymbol{k}}|}{2}\bigl|\mathscr{A}_{n}(\pm\omega/2)\bigr|^2
\bigl|\mathcal{A}_{m}(\pm\omega)\bigr|^{2\Bbbk}\bigl|\mathrm{e}^{i\omega/2}-1\bigr|^{n+1}\bigl|\hat{B}_{n;\boldsymbol{s}}(\omega/2)\bigr|.\end{multline}  

It follows from \eqref{phi_fint} and \eqref{Npsihat'}{, respectively,} that
\begin{equation*}\label{dr}\mathbf{\Phi}_{n,\boldsymbol{k}}(x)={\beta}_{n}\,B_{n,\boldsymbol{k}}(x),\end{equation*}
\begin{equation*}
\mathbf{\Psi}_{n,\boldsymbol{k},\boldsymbol{s}}(x)=\frac{{\gamma}_{n,\boldsymbol{k}}}{2^{n}}
\Bigl[
\sum_{j=0}^{n}\frac{\lambda_{j}(n)}{2(-1)^{j}}
\bigl[B_{2n+1}^{(n+1)}\bigl(2(x-\boldsymbol{s})+n+j\bigr)+B_{2n+1}^{(n+1)}\bigl(2(x-\boldsymbol{s})+n-j\bigr)\Bigr]
.\end{equation*}

Notice that $\mathbf{\Phi}_{n,\boldsymbol{k}}=\sum_{\kappa=0}^n\boldsymbol{\alpha}'_\kappa\cdot \phi_{n,\boldsymbol{k}-\kappa}$ and $\mathbf{\Psi}_{n,m(\Bbbk);\boldsymbol{k},\boldsymbol{s}}=\sum_{|\kappa|\le n+m\Bbbk}\boldsymbol{\alpha}''_\kappa\cdot \psi_{n,\boldsymbol{k},\boldsymbol{s}+\kappa}$, where 
\begin{equation}\label{positive0}\sum_{\kappa=0}^n\boldsymbol{\alpha}'_\kappa=\bigl(1+r_1(n)\bigr)\ldots\bigl(1+r_n(n)\bigr):=\mathbf{\Lambda}'_n>0,\end{equation}
\begin{multline}\label{positive}\sum_{|\kappa|\le n+m\Bbbk}\boldsymbol{\alpha}''_\kappa=\Bigl[\bigl(1+r_1(n)\bigr)\bigl(1-r_1^2(n)\bigr)\ldots\bigl(1+r_n(n)\bigr)\bigl(1-r_n^2(n)\bigr)\Bigr]\\\Bigl[\bigl(1-r_1^2(m)\bigr)\ldots\bigl(1-r_m^2(m)\bigr)\Bigr]^{2\Bbbk}:=\mathbf{\Lambda}''_n>0\end{multline} (see \eqref{phi_fint}, \eqref{Npsihat'} and definitions of $\mathbf{A}_n(\pm\omega)$ and $\mathcal{A}_n(\pm\omega)$).
The $\mathbf{\Phi}_{n,\boldsymbol{k}}$ and $\mathbf{\Psi}_{n,m(\Bbbk);\boldsymbol{k},\boldsymbol{s}}$ are localised (e.g. have compact supports), in addition, \begin{equation}\label{vazhno}\textrm{supp}\,\mathbf{\Phi}_{n,\boldsymbol{k}}=[\boldsymbol{k},\boldsymbol{k}+n+1]\quad\textrm{and}\quad\textrm{supp}\,\mathbf{\Psi}_{n,m(\Bbbk);\boldsymbol{k},\boldsymbol{s}}=[\boldsymbol{s}-n/2-m\Bbbk,\boldsymbol{s}+3n/2+m\Bbbk+1].\end{equation} It follows from the localisation algorithms that the functions $\mathbf{\Phi}_{n,\boldsymbol{k}}$ and $\mathbf{\Psi}_{n,m(\Bbbk);\boldsymbol{k},\boldsymbol{s}}$ are finite linear combinations of integer translations of $\phi_{n,\boldsymbol{k}}$ and $\psi_{n,\boldsymbol{k},\boldsymbol{s}}$, respectively, which are elements of the same orthonormal basis in $\textrm{MRA}_{B_n}$ of $L_2(\mathbb{R})$. On the strength of \eqref{Fu0} the system $\{\mathbf{\Phi}_{n,\boldsymbol{k}}(\cdot-\tau)\colon \tau\in\mathbb{Z}\}$ forms a Riesz basis in the subspace $\mathscr{V}_0\subset L_2(\mathbb{R})$ related to $\textrm{MRA}_{B_n}$. 
Integer translates of $\mathbf{\Psi}_{n,m(\Bbbk);\boldsymbol{k},\boldsymbol{s}}$ also form a Riesz basis in $\mathscr{W}_0\subset L_2(\mathbb{R})$ related to $\textrm{MRA}_{B_n}$. The both facts are confirmed by the forms of $|\hat{\mathbf{\Phi}}_{n,\boldsymbol{k}}|$ and $|\hat{\mathbf{\Psi}}_{n,m(\Bbbk);\boldsymbol{k},\boldsymbol{s}}|$ (see \eqref{Fu0}, \eqref{Fu}). Further properties of $\bigl\{\mathbf{\Phi}_{n,\boldsymbol{k}},\mathbf{\Psi}_{n,m(\Bbbk);\boldsymbol{k},\boldsymbol{s}}\bigr\}$ can be found in \cite[Proposition 3.2]{RMC}.


\section{Characterisation of weighted function spaces by Battle--Lemari\'{e} spline wavelets}\label{Char} 


Let $C(\mathbb{R}^N)$ be the space of all complex--valued uniformly continuous bounded functions in $\mathbb{R}^N$ and let $$C^M(\mathbb{R}^N)=\bigl\{f\in C(\mathbb{R}^N)\colon D^\gamma f\in C(\mathbb{R}^N),\,|\gamma|\le M\bigr\},\qquad M\in\mathbb{N}_0,$$ obviously normed. We shall use the convention $C^0(\mathbb{R}^N)=C(\mathbb{R}^N)$.
  
For a given $N\in\mathbb{N}$ and for each $l\in\{1,\ldots,N\}$ let $\mathscr{V}_{d}$, $d\in\mathbb{N}_0$, denote the multi--resolution approximation of $L_2(\mathbb{R})$ generated by $B_{n_l}$ with $\mathscr{V}_{d+1}=\mathscr{V}_{d}\oplus\mathscr{W}_{d}$ for each $d\in\mathbb{N}_0$ (see p. \pageref{MRA}). Define $V_{d}$, $d\in\mathbb{N}_0$, as the closure in $L_2(\mathbb{R}^N)-$norm of the tensor product
$\underbrace{\mathscr{V}_{d}\otimes\ldots\otimes\mathscr{V}_{d}}_{N \textrm{ times}}$, $d\in\mathbb{N}_0$. 
 
For each $l\in\{1,\ldots,N\}$ we fix $n_l,m_l\in\mathbb{N}$, $\Bbbk_l\in\{0,1\}$ with $\boldsymbol{k}_l$, $\boldsymbol{s}_l\in\mathbb{Z}$ and denote \begin{equation}\label{vazhnoo}\widetilde{\mathbf{\Phi}}(x):=\mathbf{\Phi}_{n_l,\boldsymbol{k}_l}(x)/\mathbf{\Lambda}'_{n_l}\quad\textrm{and}\quad \widetilde{\mathbf{\Psi}}(x):=\mathbf{\Psi}_{n_l,m_l(\Bbbk_l);\boldsymbol{k}_l,\boldsymbol{s}_l}(x)/\mathbf{\Lambda}''_{n_l}\end{equation} with $\mathbf{\Lambda}'_{n_l}$ and $\mathbf{\Lambda}''_{n_l}$ as in \eqref{positive0} and \eqref{positive}. Therefore, we have chosen $N$ one--dimensional Battle--Lemari\'{e} wavelet systems $\bigl\{\widetilde{\mathbf{\Phi}},\widetilde{\mathbf{\Psi}}\bigr\}$ of order $n_l$. Applying the tensor--product proce\-du\-re one can obtain the following $N-$dimensional scaling and wavelet functions $\mathbf{\Phi}$ and $\mathbf{\Psi}$, with compact supports on $\mathbb{R}^N$, related to the pairs $\{\widetilde{\mathbf{\Phi}},\widetilde{\mathbf{\Psi}}\}$:
\begin{equation}\label{wavelets_main}{\mathbf{\Phi}}\in C^{n_0-1}(\mathbb{R}^N),\qquad {\mathbf{\Psi}}_{i}\in C^{n_0-1}(\mathbb{R}^N)\quad(i=1,\ldots,2^N-1),\end{equation} here $n_0=\min\{n_1,\ldots,n_N\}$. Being linear combinations of basis functions (from $\mathscr{V}_0$ only or $\mathscr{W}_0$ only), integer translates of $\widetilde{\mathbf{\Phi}}$ and $\widetilde{\mathbf{\Psi}}$ form (semi--orthogonal) Riesz bases (in $\mathscr{V}_0$ and $\mathscr{W}_0$, respectively) of order $n_l$ for each $l\in\{1,\ldots,N\}$. 
Thus, ${\mathbf{\Phi}}$ and ${\mathbf{\Psi}}$ constitute Riesz basss in ${V}_0$ and ${W}_0$ of smoothness $n_0-1$. 
 
For $x\in\mathbb{R}^N$ we denote \begin{multline}\label{ForRepr'}{\mathbf{\Phi}}_{\tau}(x):=\mathbf{\Phi}(x-\tau)\quad\textrm{and}\quad {\mathbf{\Psi}}_{id\tau}(x):=2^{dN/2}\mathbf{\Psi}_i(2^dx-\tau) \\ (i=1,\ldots,2^N-1;\,\tau\in\mathbb{Z}^N;\, d\in\mathbb{N}_0).\end{multline}


Basis functions ${\mathbf{\Phi}}_{\tau}$ and ${\mathbf{\Psi}}_{id\tau}$ from wavelet systems \eqref{ForRepr'} generated by \eqref{vazhnoo} and \eqref{wavelets_main} can both serve as atoms and local means (see Definitions \ref{Atoms} and \ref{localmeans} in \S\,\ref{ddd}).

\subsection{Atomic representation}\label{ddd}
For $(\tau_1,\ldots,\tau_N)=\tau\in\mathbb{Z}^N$ and $d\in\mathbb{N}_0$ we define dyadic cubes 
$$Q_{d\tau}:=\prod_{l=1}^N\Bigl[\frac{\tau_l-1/2}{2^d},\frac{\tau_l+1/2}{2^d}\Bigr]$$ with sides of lengths $2^{-d}$ parallel to the axes of coordinates. For $0<p<\infty$, $d\in\mathbb{N}_0$ and $\tau\in\mathbb{Z}^N$ we denote by $\chi_{d\tau}^{(p)}$ the $p-$normalised characteristic function of the cube 
$Q_{d\tau}$: $$\chi_{d\tau}^{(p)}(x):=2^{dN/p}
\chi_{d\tau}(x):=\begin{cases}2^{dN/p}, & x\in Q_{d\tau},\\ 0, & x\not\in Q_{d\tau},\end{cases}\qquad \|\chi_{d\tau}^{(p)}\|_{L^p(\mathbb{R}^N)}=1.$$
\begin{definition}\label{Atoms}
{\rm Let $-\infty<s<+\infty$, $0<p<\infty$, $\mathbf{K},\mathbf{L}\in \mathbb{N}_0$ and $\boldsymbol{d}\ge 1$.\\ (i) Functions $a\in C^\mathbf{K}(\mathbb{R}^N)$ are called $1_\mathbf{K}-$atoms if $\textrm{supp}\,a\subset \boldsymbol{d}\,Q_{0\tau}$ for some $\tau\in\mathbb{Z}^N$ and $|D^{\boldsymbol{\alpha}} a(x)|\le 1$ for $|\boldsymbol{\alpha}|\le \mathbf{K}$, $x\in\mathbb{R}^N$.\\
(ii) Functions $a\in C^\mathbf{K}(\mathbb{R}^N)$ are called $(s,p)_{\mathbf{K},\mathbf{L}}-$atoms if for some $d\in\mathbb{N}$ $$\textrm{supp}\,a\subset \boldsymbol{d}\,Q_{d\tau}\qquad\textrm{with some}\quad \tau\in\mathbb{Z}^N,$$ $$|D^{\boldsymbol{\alpha}} a(x)|\le 2^{-d(s-N/p)+d|{\boldsymbol{\alpha}}|}\qquad \textrm{for}\quad |{\boldsymbol{\alpha}}|\le \mathbf{K},\ x\in\mathbb{R}^N,$$ $$\int_{\mathbb{R}^N}x^{\boldsymbol{\beta}} a(x)\,dx=0\qquad \textrm{for}\quad |{\boldsymbol{\beta}}|< \mathbf{L}. $$
}\end{definition}

\begin{definition} {\rm Let $w\in\mathscr{A}_\infty^\loc$. For $0<p<\infty$, $0<q\le\infty$ we define
a weighted space $\boldsymbol{b}_{pq}^w$ to be the collection of all sequences $\boldsymbol{\lambda}$ given by \begin{equation}\label{sequences}\boldsymbol{\lambda}=\{\boldsymbol{\lambda}_{00\tau}\}_{\tau\in\mathbb{Z}^N}\cup\{\boldsymbol{\lambda}_{id\tau}\}_{{i=1,\ldots,2^N-1};\,{d\in\mathbb{N};\,\tau\in\mathbb{Z}^N}}\quad\textrm{with}\quad \boldsymbol{\lambda}_{00\tau},\boldsymbol{\lambda}_{id\tau}\in\mathbb{C}\end{equation} such that
\begin{multline*}\|\boldsymbol{\lambda}\|_{\boldsymbol{b}_{pq}^w}:=\biggl(\int_{\mathbb{R}^N}\Bigl|\sum_{\tau\in\mathbb{Z}^N}\boldsymbol{\lambda}_{00\tau}\chi_{0\tau}^{(p)}(x)\Bigr|^pw(x)\,dx\biggr)^{\frac{1}{p}}\\+
\Biggl(\sum_{d=1}^\infty \sum_{i=1}^{2^N-1}\biggl(\int_{\mathbb{R}^N}\Bigl|\sum_{\tau\in\mathbb{Z}^N} \boldsymbol{\lambda}_{id\tau}\chi_{d\tau}^{(p)}(x)\Bigr|^p w(x)\,dx\biggr)^\frac{q}{p}\Biggr)^\frac{1}{q}<\infty.\end{multline*} For $0<p<\infty$, $0<q\le\infty$ or $p=q=\infty$ a weighted space $\boldsymbol{f}_{pq}^w$ is the collection of all \eqref{sequences} such that
\begin{multline*}\|\boldsymbol{\lambda}\|_{\boldsymbol{f}_{pq}^w}=
\biggl(\int_{\mathbb{R}^N}\biggl(\sum_{\tau\in\mathbb{Z}^N} \Bigl|\boldsymbol{\lambda}_{00\tau}\chi_{0\tau}^{(p)}(x)\Bigr|^{q}\biggr)^\frac{p}{q}w(x)\,dx\biggr)^{\frac{1}{p}}\\
+\Biggl(\int_{\mathbb{R}^N}\biggl(\sum_{d=1}^\infty \sum_{i=1}^{2^N-1}\sum_{\tau\in\mathbb{Z}^N}\Bigl|\boldsymbol{\lambda}_{id\tau}\chi_{d\tau}^{(p)}(x)\Bigr|^q\biggr)^\frac{p}{q}w(x)\,dx\Biggr)^{\frac{1}{p}}<\infty.\end{multline*}}\end{definition} The expressions defining the (quasi--)norms $\|\cdot\|_{\boldsymbol{b}_{pq}^w}$ and $\|\cdot\|_{\boldsymbol{f}_{pq}^w}$ are subjects of standard modification in the cases of infinite parameters $p$ and/or $q$.

To simplify the notation we write $\boldsymbol{a}_{pq}^w$ instead of either $\boldsymbol{b}_{pq}^w$ or $\boldsymbol{f}_{pq}^w$.
The sequence spaces $\boldsymbol{a}_{pq}^w$ are adapted versions of the spaces from \cite[Definition 3.2]{IS} or \cite[Definition 4.2]{WBan} with respect to the additional parameter $i=1,\ldots,2^N-1$.
We shall denote an atom $a(x)$ supported in some $Q_{d\tau}$ by $a_{id\tau}$ in the sequel.

For $w\in\mathscr{A}_\infty^\loc$ with $\boldsymbol{r}_0$ of the form \eqref{r_w} we define $$\sigma_{p}(w):=N\Bigl(\frac{\boldsymbol{r}_0}{\min\{p,\boldsymbol{r}_0\}}-1\Bigr)+N(\boldsymbol{r}_0-1),$$ $$\sigma_q:=\frac{N}{\min\{1,q\}}-N\qquad\textrm{and}\qquad \sigma_{pq}(w):=\max\bigl\{\sigma_p(w),\sigma_q\bigr).$$

In 2012 Izuki M. and Sawano Y. proved the atomic representation theorem for elements of spaces $A_{pq}^{s,w}(\mathbb{R}^N)$ \cite{IS}. This result admits the following reformulation (see \cite[Theorem 3.4]{IS} and also \cite[Theorem 4.3]{WBan}).

\begin{theorem}\label{AtomTh}
Let $0<p<\infty$, $0<q\le\infty$, $-\infty<s<+\infty$ and $w\in\mathscr{A}_\infty^\loc$. Assume that $\mathbf{K},\mathbf{L}+1\in\mathbb{N}_0$ satisfy the conditions $$\mathbf{K}\ge (1+[s])_+\qquad \textrm{and}\qquad \mathbf{L}\ge \max\bigl\{-1,[\sigma_{p}(w)-s]\bigr\}$$ when $A_{pq}^{s,w}(\mathbb{R}^N)$ denotes $B_{pq}^{s,w}(\mathbb{R}^N)$, or are such that
$$\mathbf{K}\ge (1+[s])_+\qquad \textrm{and}\qquad \mathbf{L}\ge \max\bigl\{-1,[\sigma_{p,q}(w)-s]\bigr\}$$ in the case $A_{pq}^{s,w}(\mathbb{R}^N)=F_{pq}^{s,w}(\mathbb{R}^N)$.\\ {\rm(i)} If $f\in A_{pq}^{s,w}(\mathbb{R}^N)$ then there exist sequences of {$1_\mathbf{K}-$}atoms $\{\boldsymbol{a}_{00\tau}\}$, ${\tau\in\mathbb{Z}^N}$, and {$(s,p)_{\mathbf{K},\mathbf{L}}-$atoms} $\{\boldsymbol{a}_{id\tau}\}$, ${i=1,\ldots,2^N-1}$, $d\in\mathbb{N}$, $\tau\in\mathbb{Z}^N$, and $\boldsymbol{\lambda}\in \boldsymbol{a}_{pq}^w$ such that $$f=\sum_{\tau\in\mathbb{Z}^N}\boldsymbol{\lambda}_{00\tau}\,\boldsymbol{a}_{00\tau}+ \sum_{d=1}^\infty\sum_{i=1}^{2^N-1}\sum_{\tau\in\mathbb{Z}^N}\boldsymbol{\lambda}_{id\tau}\,\boldsymbol{a}_{id\tau}$$ with convergence in $\mathscr{S}'_e(\mathbb{R}^N)$, and $\|\boldsymbol{\lambda}\|_{\boldsymbol{a}_{pq}^w}\le c\|f\|_{A_{pq}^{s,w}(\mathbb{R}^N)}$.\\  
{\rm(ii)} Conversely, if $\{\boldsymbol{a}_{00\tau}\}$, ${\tau\in\mathbb{Z}^N}$, and $\{\boldsymbol{a}_{id\tau}\}$, ${i=1,\ldots,2^N-1}$, $d\in\mathbb{N}$, $\tau\in\mathbb{Z}^N$, are sequences of {$1_\mathbf{K}-$ and $(s,p)_{\mathbf{K},\mathbf{L}}-$atoms, respectively,} and $\boldsymbol{\lambda}\in \boldsymbol{a}_{pq}^w$ then the series 
$$f=\sum_{\tau\in\mathbb{Z}^N}\boldsymbol{\lambda}_{00\tau}\,\boldsymbol{a}_{00\tau}+ \sum_{d=1}^\infty\sum_{i=1}^{2^N-1}\sum_{\tau\in\mathbb{Z}^N}\boldsymbol{\lambda}_{id\tau}\,\boldsymbol{a}_{id\tau}$$ converges in $\mathscr{S}'_e(\mathbb{R}^N)$ and belongs to $A_{pq}^{s,w}(\mathbb{R}^N)$. Moreover, $$\|f\|_{A_{pq}^{s,w}(\mathbb{R}^N)}\le c\|\boldsymbol{\lambda}\|_{\boldsymbol{a}_{pq}^w
}.$$ \end{theorem}

\subsection{Characterisation by local means}
Following \cite{TBan} and \cite{WBan} we start from definitions of special classes of functions of finite smoothness (kernels) and from notion of distributions of finite order.
\begin{definition}\label{kernels}
{\rm Let $\mathbf{A},\mathbf{B}\in \mathbb{N}_0$ and $\mathbf{C}> 0$. Then {$C^\mathbf{A}(\mathbb{R}^N)-$}functions $$k_{00\tau}\colon\mathbb{R}^N\mapsto\mathbb{C}\quad\textrm{�}\ \tau\in\mathbb{Z}^N\quad\textrm{and}\ k_{id\tau}\colon\mathbb{R}^N\mapsto\mathbb{C}\quad\textrm{with}\ i=1,\ldots,2^N-1,\,d\in\mathbb{N},\,\tau\in\mathbb{Z}^N$$ are called kernels if, firstly, for all $i=0,\ldots,2^N-1$ $$\textrm{supp}\,k_{id\tau}\subset \mathbf{C}\,Q_{d\tau},\qquad d\in\mathbb{N}_0,\quad \tau\in\mathbb{Z}^N;$$ secondly, there exist all (classical) derivatives $D^{\boldsymbol{\alpha}} k_{id\tau}\in C(\mathbb{R}^N)$ with $|{\boldsymbol{\alpha}}|\le\mathbf{A}$ such that for all $i=0,\ldots,2^N-1$ $$|D^{\boldsymbol{\alpha}} k_{id\tau}(x)|\le 2^{dN+d|{\boldsymbol{\alpha}}|},\qquad |{\boldsymbol{\alpha}}|\le \mathbf{A},\quad d\in\mathbb{N}_0,\quad \tau\in\mathbb{Z}^N,$$ and $$\int_{\mathbb{R}^N}x^{\boldsymbol{\beta}} k_{id\tau}(x)\,dx=0,\qquad |{\boldsymbol{\beta}}|< \mathbf{B},\quad d\in\mathbb{N},\quad \tau\in\mathbb{Z}^N.$$
}\end{definition}

{Let $C_K^M(\mathbb{R}^N)$ be a set of functions $\varphi\in C^M(\mathbb{R}^N)$ such that ${\rm supp}\,\varphi\subset K$, where $K\subset\mathbb{R}^N$ is a compact.} {By $C_0^M(\mathbb{R}^N)$ we denote a set consisting of all functions of order $M$ with compact supports.} 

\begin{definition} {\rm A distribution $f\in\mathscr{D}'(\mathbb{R}^N)$ is of order $M$ if for every compact $K\subset \mathbb{R}^N$ there exists a constant $c$ such that
$$|\langle f,\varphi\rangle|\le c\sum_{|\gamma|\le M} \sup
_{x\in K}|D^\gamma\varphi(x)|\quad \textrm{for every} \ 
\varphi\in C_0^\infty(\mathbb{R}^N).$$ {The set of all distributions of order $M$ is denoted by $\mathscr{D}'_M(\mathbb{R}^N).$ If $f\in\mathscr{D}'_M(\mathbb{R}^N)$ then $f$ can be extended to a} continuous linear functional on $C_0^\infty(\mathbb{R}^N)$. Moreover \cite{Ho}, $$(C_0^M(\mathbb{R}^N))'=\mathscr{D}'_M(\mathbb{R}^N).$$}\end{definition}

\begin{remark}\label{dfo}It was proven in \cite[Lemma 5.6 and Corollary 5.7]{WBan} that $A_{pq}^{s,w}(\mathbb{R}^N)$ consists of distributions of finite order $M$ provided $M\ge\bigl([N(\boldsymbol{r}_0-1)/p-s]+1\bigr)_+$.\end{remark}

Further, following again the concept from \cite{TBan} and \cite[\S~5]{WBan}, we define local means and introduce sequence spaces with local Muckenhoupt weights adaptively to the parameter $i=0,\ldots,2^N-1$.
\begin{definition}\label{localmeans} {\rm {Let $f\in
\mathscr{D}'_\mathbf{A}(\mathbb{R}^N)\cap\mathscr{S}'_e(\mathbb{R}^N)$.} Let $k_{id\tau}$ be kernels according to Definition \ref{kernels} {(with the same $\mathbf{A}$)}. Then $k_{00\tau}(f)=\langle f,k_{00\tau}\rangle$ with $\tau\in\mathbb{Z}^N$ and $k_{id\tau}(f)=\langle f,k_{id\tau}\rangle$ with $i=1,\ldots,2^N-1$, $d\in\mathbb{N}$, $\tau\in\mathbb{Z}^N$, are called local means. Furthermore, we put $$k(f)=\bigl\{k_{00\tau}(f)\colon \tau\in\mathbb{Z}^N\bigr\}\cup\bigl\{k_{id\tau}(f)\colon i=1,\ldots,2^N-1, d\in\mathbb{N},\, \tau\in\mathbb{Z}^N\bigr\}.$$
}\end{definition}

\begin{definition}{\rm Let $-\infty<s<+\infty$, $0<p<\infty$, $0<q\le\infty$ and $w\in\mathscr{A}_\infty$. Then $b_{pq}^{s,w}$ is the collection of all \begin{equation}\label{sequences'}\lambda=\{\lambda_{00\tau}\}_{\tau\in\mathbb{Z}^N}\cup\{\lambda_{id\tau}\}_{{i=1,\ldots,2^N-1};\,{d\in\mathbb{N};\,\tau\in\mathbb{Z}^N}}\quad\textrm{with}\quad \lambda_{00\tau},\lambda_{id\tau}\in\mathbb{C}\end{equation} such that \begin{multline}\label{0404}\|\lambda\|_{{b}^{s,w}_{pq}}=\biggl(\int_{\mathbb{R}^N}\Bigl(\sum_{\tau\in\mathbb{Z}^N}\bigl|\lambda_{00\tau}\bigr|\chi_{0\tau}^{(p)}(x)\Bigr)^pw(x)\,dx\biggr)^{\frac{1}{p}}\\+
\Biggl(\sum_{d=1}^\infty 2^{d(s-N/p)q}\sum_{i=1}^{2^N-1}\biggl(\int_{\mathbb{R}^N}\Bigl(\sum_{\tau\in\mathbb{Z}^N} \bigl|\lambda_{id\tau}\bigr|\chi_{d\tau}^{(p)}(x)\Bigr)^p w(x)\,dx\biggr)^\frac{q}{p}\Biggr)^\frac{1}{q}<\infty,\end{multline} and $f_{pq}^{s,w}$ is the collection of all sequences of the form \eqref{sequences'} satisfying the condition
\begin{multline}\label{0404'}\|\lambda\|_{{f}^{s,w}_{pq}}=
\biggl(\int_{\mathbb{R}^N}\Bigl(\sum_{\tau\in\mathbb{Z}^N}\Bigr| \lambda_{00\tau}\chi_{0\tau}^{(p)}(x)\Bigr|^q\biggr)^\frac{p}{q}w(x)\,dx\biggr)^{\frac{1}{p}}\\
+\Biggl(\int_{\mathbb{R}^N}\biggl(\sum_{d=1}^\infty 2^{dsq}\sum_{i=1}^{2^N-1}\sum_{\tau\in\mathbb{Z}^N}\Bigr| \lambda_{id\tau}\chi_{d\tau}^{(p)}(x)\Bigr|^q\biggr)^\frac{p}{q}w(x)\,dx\Biggr)^{\frac{1}{p}}<\infty.\end{multline} The (quasi--)norms \eqref{0404} and \eqref{0404'} should be modified in a standard way in the case $q=\infty$.}\end{definition}

Characterisation of function spaces $A_{pq}^{s,w}(\mathbb{R}^N)$ with $w\in\mathscr{A}_\infty^\loc$ by local means was given in \cite[Theorem 5.8]{WBan}. This important result admits the following reformulation. 
\begin{theorem}\label{LocTh}
Let $0<p<\infty$, $0<q\le\infty$ and $-\infty<s<+\infty$. Assume $w\in\mathscr{A}_\infty$. Suppose that $k_{id\tau}$ are kernels according to Definition \ref{kernels} with fixed $\mathbf{C}>0$ and $\mathbf{A},\mathbf{B}\in\mathbb{N}_0$ satisfying the conditions $$\mathbf{A}\ge \max\Bigl\{0,[\sigma_{p}(w)-s],{[N(\boldsymbol{r}_0-1)/p-s]}+1\Bigr\}\qquad \textrm{and}\qquad \mathbf{B}\ge (1+[s])_+$$ if $A_{pq}^{s,w}(\mathbb{R}^N)=B_{pq}^{s,w}(\mathbb{R}^N)$, or
$$\mathbf{A}\ge \max\Bigl\{0,[\sigma_{pq}(w)-s],{[N(\boldsymbol{r}_0-1)/p-s]}+1\Bigr\}\qquad \textrm{and}\qquad \mathbf{B}\ge (1+[s])_+$$ when $A_{pq}^{s,w}(\mathbb{R}^N)$ denotes $F_{pq}^{s,w}(\mathbb{R}^N)$. Let $k(f)$ be as in Definition \ref{localmeans}. Then for some $c>0$ and all $f\in A_{pq}^{s,w}(\mathbb{R}^N)$ $$\|k(f)\|_{{a}_{pq}^{s,w}}\le c\|f\|_{A_{pq}^{s,w}(\mathbb{R}^N)},$$ where ${a}^{s,w}_{pq}$ stands for either ${b}_{pq}^{s,w}$ or ${f}_{pq}^{s,w}$, respectively to the meaning of $A_{pq}^{s,w}(\mathbb{R}^N)$. \end{theorem}


For proving our main result we shall need the local reproducing formula  for $f\in A_{pq}^{s,w}(\mathbb{R}^N)$ by V.S. Rychkov \cite{R} utilising the idea from \cite{FJ}. Denote by $\{\boldsymbol{\varphi}_d\}_{d\in\mathbb{N}_0}$ the family of functions from $\mathscr{D}(\mathbb{R}^N)$ satisfying \eqref{Iphi1}--\eqref{Iphi0} with $\Gamma=(1+[s])_+$. It was established in \cite[Theorem 1.6]{R} that for any given integer $\mathbf{L}\ge 0$ there exists another family of functions $\{\boldsymbol{\psi}_d\}_{d\in\mathbb{N}_0}\subset\mathscr{D}(\mathbb{R}^N)$ such that $\boldsymbol{\psi}_1$ has at least $\mathbf{L}\ge 0$ vanishing moments and 
\begin{equation}\label{arep}
f(x)=\sum_{d\in\mathbb{N}_0}\boldsymbol{\psi}_d\ast\boldsymbol{\varphi}_d\ast f
\qquad\textrm{in }\mathscr{D}'(\mathbb{R}^N),\quad\textrm{for all }f\in \mathscr{D}'(\mathbb{R}^N).\end{equation}

\subsection{The main theorem}\label{mn}
We shall represent a distribution $f\in A_{pq}^{s,w}(\mathbb{R}^N)$ as the series
\begin{equation*}
f=\sum_{\tau\in\mathbb{Z}^N}\langle f,\mathbf{\Phi}_{\tau}\rangle \mathbf{\Phi}_\tau+
\sum_{d\in\mathbb{N}}\sum_{i=1}^{2^N-1}\sum_{\tau\in\mathbb{Z}^N}\langle f,\mathbf{\Psi}_{i(d-1)\tau}\rangle\mathbf{\Psi}_{i(d-1)\tau}\end{equation*} 
converging in $\mathscr{S}'_e(\mathbb{R}^N)$.
Our main result reads as follows.
 
\begin{theorem}\label{main}
Let $0<p<\infty$, $0<q\le\infty$, $-\infty<s<+\infty$ and $w\in\mathscr{A}_\infty^\loc$. Let 
${\mathbf{\Phi}},\, {\mathbf{\Psi}}_{i}\in C^{n_0-1}(\mathbb{R}^N)$ with $i=1,\ldots,2^N-1$ be functions satisfying \eqref{wavelets_main}. We assume \begin{equation}\label{condB}n_0\ge \max\Bigl\{0,[s]+1,{\bigl[N(\boldsymbol{r}_0-1)/p-s\bigr]}+1,\bigl[\sigma_{p}(w)-s\bigr]\Bigr\}+1\end{equation} in the case $A_{pq}^{s,w}(\mathbb{R}^N)=B_{pq}^{s,w}(\mathbb{R}^N)$, and
\begin{equation}\label{condF}n_0\ge \max\Bigl\{0,[s]+1,{\bigl[N(\boldsymbol{r}_0-1)/p-s\bigr]}+1,\bigl[\sigma_{pq}(w)-s\bigr]\Bigr\}+1\end{equation} when $A_{pq}^{s,w}(\mathbb{R}^N)$ denotes $F_{pq}^{s,w}(\mathbb{R}^N)$. Then $f\in\mathscr{S}'_e(\mathbb{R}^N)$ belongs to $A_{pq}^{s,w}(\mathbb{R}^N)$ if and only if it can be represented as
\begin{equation*}
f=\sum_{\tau\in\mathbb{Z}^N} \lambda_{00\tau}\mathbf{\Phi}_\tau+
\sum_{d\in\mathbb{N}}\sum_{i=1}^{2^N-1}\sum_{\tau\in\mathbb{Z}^N}\lambda_{id\tau}2^{-dN/2}\mathbf{\Psi}_{i(d-1)\tau},
\end{equation*} where $\lambda\in a_{pq}^{s,w}$ and the series converges in $\mathscr{S}'_e(\mathbb{R}^N)$. This representation is unique with
$$\lambda_{00\tau}=\langle f,\mathbf{\Phi}_{\tau}\rangle\quad(\tau\in\mathbb{Z}^N),\qquad   \lambda_{id\tau}=2^{dN/2}\langle f,\mathbf{\Psi}_{i(d-1)\tau}\rangle\quad(i=1,\ldots,2^N-1;\,d\in\mathbb{N};\,\tau\in\mathbb{Z}^N)$$
and $I\colon f\mapsto\Bigl\{\bigl\{\lambda_{00\tau}\bigr\}\cup\bigl\{\lambda_{id\tau}\bigr\}\Bigr\}$ is a linear isomorphism of $A_{pq}^{s,w}(\mathbb{R}^N)$ onto $a_{pq}^{s,w}$. Besides, \begin{equation}\label{isom}(i)\ \|\lambda\|_{a_{pq}^{s,w}}\lesssim
\|f\|_{A_{pq}^{s,w}(\mathbb{R}^N)}\lesssim \|\boldsymbol{\lambda}\|_{\boldsymbol{a}_{pq}^w},\qquad (ii)\ \|\boldsymbol{\lambda}\|_{\boldsymbol{a}_{pq}^w}\lesssim \|{\lambda}\|_{{a}_{pq}^{s,w}},
\end{equation} with $\boldsymbol{\lambda}_{00\tau}={\lambda}_{00\tau}$, $\tau\in\mathbb{Z}^N$, and $\boldsymbol{\lambda}_{id\tau}=2^{d\boldsymbol{\sigma}}{\lambda}_{id\tau}$, ${i=1,\ldots,2^N-1}$, ${d\in\mathbb{N}}$, $\tau\in\mathbb{Z}^N$, where $\boldsymbol{\sigma}=s-N/p$ when $\boldsymbol{a}_{pq}^w$ denotes $\boldsymbol{b}_{pq}^w$, and $\boldsymbol{\sigma}=s$ when $\boldsymbol{a}_{pq}^w$ denotes $\boldsymbol{f}_{pq}^w$.\end{theorem}

\begin{proof}[Proof] We follow the idea taken from \cite[Theorem 6.2]{WBan}.

The estimate $\|\boldsymbol{\lambda}\|_{\boldsymbol{a}_{pq}^{w}}\lesssim \|{\lambda}\|_{{a}_{pq}^{s,w}}$ is obvious.
Assume that $\|\boldsymbol{\lambda}\|_{\boldsymbol{a}_{pq}^w}<\infty$.
Let $f\in\mathscr{S}'_e(\mathbb{R}^N)$ and $$f=\sum_{\tau\in\mathbb{Z}^N} \boldsymbol{\lambda}_{00\tau}\mathbf{\Phi}_\tau+
\sum_{d\in\mathbb{N}}\sum_{i=1}^{2^N-1}\sum_{\tau\in\mathbb{Z}^N}\boldsymbol{\lambda}_{id\tau}2^{-d(s+N/2-N/p)}\mathbf{\Psi}_{i(d-1)\tau}.$$ Then
$a_{00\tau}(x)=\mathbf{\Phi}_\tau(x)$ are $1_\mathbf{K}-$atoms and $a_{id\tau}(x)=2^{-d(s+N/2-N/p)}\mathbf{\Psi}_{i(d-1)\tau}(x)$ are $(s,p)_{\mathbf{K},\mathbf{L}}-$ atoms with $\mathbf{K}=\mathbf{L}=n_0-1$. Thus, by Theorem \ref{AtomTh}, {$f\in A_{pq}^{s,w}(\mathbb{R}^N)$} and the right hand side of \eqref{isom}(i) is performed.

Now let {$f\in A_{pq}^{s,w}(\mathbb{R}^N)$}. We take ${k}_{00\tau}(x)=\mathbf{\Phi}_{\tau}(x)$, $\tau\in\mathbb{Z}^N$, and ${k}_{id\tau}(x)=2^{dN/2}\mathbf{\Psi}_{i(d-1)\tau}(x)$, ${i=1,\ldots,2^N-1}$, ${d\in\mathbb{N}}$, $\tau\in\mathbb{Z}^N$, as kernels of local means with $\mathbf{A}=\mathbf{B}=n_0-1$. Thus, the left hand side of the inequality \eqref{isom}(i) follows by Theorem \ref{LocTh}. From here and atomic decomposition
$$g=\sum_{\tau\in\mathbb{Z}^N} k_{00\tau}(f)\mathbf{\Phi}_\tau+
\sum_{d\in\mathbb{N}}\sum_{i=1}^{2^N-1}\sum_{\tau\in\mathbb{Z}^N}k_{id\tau}(f)2^{-dN/2}\mathbf{\Psi}_{i(d-1)\tau}\in A_{pq}^{s,w}(\mathbb{R}^N).$$ 
Denote $\mathbf{\Psi}_{0(-1)\tau}:=\mathbf{\Phi}_\tau$. 
By \eqref{ForRepr'}, \eqref{wavelets_main} {and \eqref{vazhnoo},} the $\mathbf{\Psi}_{i(d-1)\tau}$, $i=0,\ldots,2^{N}-1$, $d\in\mathbb{N}_{0}$, $\tau\in\mathbb{Z}^N$, are functions of $N$ variables having a form of products of $N$ functions, either $\widetilde{\mathbf{\Phi}}$ or $\widetilde{\mathbf{\Psi}}$ of one variable $x_l$, staying at $l-$th place, $l=\{1,\ldots,N\}$. Recall that each of $\widetilde{\mathbf{\Phi}}$ or $\widetilde{\mathbf{\Psi}}$ {are} finite number linear combinations of integer shifts of either ${\phi}_{n_l,\boldsymbol{k}_l}$ or ${\psi}_{n_l,\boldsymbol{k}_l,\boldsymbol{s}_l}$, which are elements of an orthogonal wavelet basis in $L_2(\mathbb{R}^N)$. 
Moreover, measures of the supports of $\widetilde{\mathbf{\Phi}}$ and $\widetilde{\mathbf{\Psi}}$ are equal to {$n_l+1$} and $2n_l+2m_l\Bbbk_l+1$, respectively (see \eqref{vazhno}). 

It follows from {Remark \ref{dfo}} that the dual pairing $\langle g,\mathbf{\Psi}_{i^\ast(d^\ast-1)\tau^\ast}\rangle$ with some $i^\ast=0,1,\ldots,$ $2^N-1$, ${d^\ast\in\mathbb{N}_0}$, $\tau^\ast\in\mathbb{Z}^N$ makes sense. Then
\begin{multline}\label{go}\langle g,\mathbf{\Psi}_{i^\ast(d^\ast-1)\tau^\ast}\rangle=
\sum_{d\in\mathbb{N}_0}\sum_{i=0}^{2^N-1}\sum_{\tau\in\mathbb{Z}^N}k_{id\tau}(f)2^{-dN/2}\langle\mathbf{\Psi}_{i(d-1)\tau},\mathbf{\Psi}_{i^\ast(d^\ast-1)\tau^\ast}\rangle\\
=\sum_{\tau\in\mathbb{Z}^N}k_{i^\ast d^\ast\tau}(f)2^{-d^\ast N/2}\langle\mathbf{\Psi}_{i^\ast(d^\ast-1)\tau},\mathbf{\Psi}_{i^\ast(d^\ast-1)\tau^\ast}\rangle=
\sum_{h\in\mathbf{Z}_{n,i^\ast}}\boldsymbol{\lambda}_{h}\langle f,\mathbf{\Psi}_{i^\ast(d^\ast-1)(\tau^\ast+h)}\rangle.\end{multline} Here $h=(h_1,\ldots,h_N)$ and, in view of product type structure of $\mathbf{\Psi}_{i^\ast(d^\ast-1)\tau^\ast}$, 
\begin{equation}\label{inn}\boldsymbol{\lambda}_{h}:=\langle \mathbf{\Psi}_{i^\ast(d^\ast-1)(\tau^\ast+h)},\mathbf{\Psi}_{i^\ast(d^\ast-1)\tau^\ast}\rangle=\boldsymbol{\lambda}_{h_1}\cdot\ldots\cdot\boldsymbol{\lambda}_{{h}_N},\quad \quad h\in\mathbf{Z}_{n,i^\ast},\end{equation} with $\boldsymbol{\lambda}_{h_l}=\langle \widetilde{\mathbf{\Phi}}_{i^\ast(d^\ast-1)(\tau_l^\ast+h_l)},\widetilde{\mathbf{\Phi}}_{i^\ast(d^\ast-1)\tau_l^\ast}\rangle$ or $\boldsymbol{\lambda}_{h_l}=\langle \widetilde{\mathbf{\Psi}}_{i^\ast(d^\ast-1)(\tau_l^\ast+h_l)},\widetilde{\mathbf{\Psi}}_{i^\ast(d^\ast-1)\tau_l^\ast}\rangle$, which depend on $r_1(n),\ldots,$ $r_n(n)$ (and, possibly, on $r_1(m),\ldots,r_m(m)$ in $\widetilde{\mathbf{\Psi}}$ case). Since $\mathbf{\Psi}_{i^\ast(d^\ast-1)\tau^\ast}$ have bounded supports, the coefficients $\boldsymbol{\lambda}_{h}$ 
are non-trivial on a bounded set $\mathbf{Z}_{n,i^\ast}=\mathbf{Z}_{(n_1,\ldots,n_N),i^\ast}$ only. Indeed, if $i^\ast=0$ then
$\mathbf{Z}_{n,0}=\{-n_1,\ldots,0,\ldots,n_1\}\times\ldots\times\{-n_N,\ldots,0,\ldots,\\n_N\}$, since each of ${\mathbf{\Psi}}_{0(-1)\tau^\ast}={\mathbf{\Phi}}_{\tau'=(\tau^\ast_1,\ldots,\tau^\ast_N)}$ is $N-$terms product of $\widetilde{\mathbf{\Phi}}_{\tau^\ast_l}$, $l=1,\ldots,N$. If $i^\ast\not =0$ then, each ${\mathbf{\Psi}}_{i^\ast(d^\ast-1)\tau^\ast}$ has either $\widetilde{\mathbf{\Phi}}_{i^\ast(d^\ast-1)\tau_l^\ast}$ with $l\in N'$ or $\widetilde{\mathbf{\Psi}}_{i^\ast(d^\ast-1)\tau_l^\ast}$ with $l\in N''$ at the $l-$th place, where $l=1,\ldots,N$ and $N'\cup N''=N$. This means that $\mathbf{Z}_{n,i^\ast}$ is $N-$terms product of $\{-n_l,\ldots,0,\ldots,n_l\}$ and/or $\{-2n_l-2m_l\Bbbk_l,\ldots,0,\ldots,2n_l+2m_l\Bbbk_l\}$. 

Therefore, inner products $\langle \mathbf{\Psi}_{i^\ast(d^\ast-1)(\tau^\ast+h)},\mathbf{\Psi}_{i^\ast(d^\ast-1)\tau^\ast}\rangle$ in \eqref{inn} have $N$ terms of the forms 
$$\boldsymbol{\lambda}_{h_l}=\langle \widetilde{\mathbf{\Phi}}_{i^\ast(d^\ast-1)(\tau_l^\ast+h_l)},\widetilde{\mathbf{\Phi}}_{i^\ast(d^\ast-1)\tau_l^\ast}\rangle\ \ (l\in N')$$ or $$\boldsymbol{\lambda}_{h_l}=\langle \widetilde{\mathbf{\Psi}}_{i^\ast(d^\ast-1)(\tau_l^\ast+h_l)},\widetilde{\mathbf{\Psi}}_{i^\ast(d^\ast-1)\tau_l^\ast}\rangle\ \ (l\in N'').$$ They are non-zero as long as $|h_l|\le n_l$ or $|h_l|\le 2n_l+2m_l\Bbbk_l$, respectively (see details of constructing $\widetilde{\mathbf{\Phi}}_{n_l,\boldsymbol{k}_l}$ and $\widetilde{\mathbf{\Psi}}_{n_l,\boldsymbol{k}_l,\boldsymbol{s}_l}$ from orthogonal to each other integer shifts of ${\phi}_{n_l,\boldsymbol{k}_l}$ and ${\psi}_{n_l,\boldsymbol{k}_l,\boldsymbol{s}_l}$ in \S~\ref{NN}). Hee we recall also that for each $l\in\{1,\ldots,N\}$ $$\widetilde{\mathbf{\Phi}}_{n_{{l}},\boldsymbol{k}_{{l}}}=\bigl(\boldsymbol{\Lambda}'_{n_{{l}}}\bigr)^{-1}\sum_{\kappa_{{l}}=0}^{n_{{l}}}\boldsymbol{\alpha}'_{\kappa_{{l}}}\cdot \phi_{n_{{l}},\boldsymbol{k}_{{l}}-\kappa_{{l}}}$$ and $$\widetilde{\mathbf{\Psi}}_{n_{{l}},m_{{l}}(\Bbbk_{{l}});\boldsymbol{k}_{{l}},\boldsymbol{s}_{{l}}}=\bigl(\boldsymbol{\Lambda}''_{n_{{l}}}\bigr)^{-1}\sum_{|\kappa_{{l}}|\le n_{{l}}+m_{{l}}\Bbbk_{{l}}}\boldsymbol{\alpha}''_{\kappa_{{l}}}\cdot \psi_{n_{{l}},\boldsymbol{k}_{{l}},\boldsymbol{s}_{{l}}+\kappa_{{l}}}.$$ Therefore, for a fixed set of $l\in\{1,\ldots,N\}\setminus\boldsymbol{l}$ we have in the case $\boldsymbol{l}\in N'$, taking into account \eqref{positive0} and the fact that each of $\widetilde{\mathbf{\Phi}}$ is a combination of orthogonal to each other $n_{\boldsymbol{l}}+1$ sequential elements $\phi_{n_{\boldsymbol{l}},\boldsymbol{k}_{\boldsymbol{l}}}$: 
\begin{multline*}\sum_{h_{{\boldsymbol{l}}}}\boldsymbol{\lambda}_{h_{{\boldsymbol{l}}}}=
\sum_{|h_{\boldsymbol{l}}|\le n_{\boldsymbol{l}}}\boldsymbol{\lambda}_{h_{{\boldsymbol{l}}}}=\bigl(\mathbf{\Lambda}'_{n_{\boldsymbol{l}}}\bigr)^{-2}\biggl[\sum_{\kappa_{\boldsymbol{l}}=0}^{ n_{\boldsymbol{l}}}\bigl(\boldsymbol{\alpha}'_{\kappa_{\boldsymbol{l}}}\bigr)^2+
\sum_{|\kappa^\ast_{\boldsymbol{l}}-\kappa^\star_{\boldsymbol{l}}|= 1}\boldsymbol{\alpha}'_{\kappa^\ast_{\boldsymbol{l}}}\cdot \boldsymbol{\alpha}'_{\kappa^\star_{\boldsymbol{l}}}\\+\ldots+\sum_{|\kappa^\ast_{\boldsymbol{l}}-\kappa^\star_{\boldsymbol{l}}|= n_{\boldsymbol{l}}}\boldsymbol{\alpha}'_{\kappa^\ast_{\boldsymbol{l}}}\cdot \boldsymbol{\alpha}'_{\kappa^\star_{\boldsymbol{l}}}\biggr]=
\bigl(\mathbf{\Lambda}'_{n_{\boldsymbol{l}}}\bigr)^{-2}\biggl[\sum_{\kappa_{\boldsymbol{l}}=0}^{ n_{\boldsymbol{l}}}\boldsymbol{\alpha}'_{\kappa_{\boldsymbol{l}}}\biggr]^2=1.\end{multline*}
Thus, we obtain, by grouping, if $i^\ast=0$:
\begin{equation}\label{La'}\sum_{h\in\mathbf{Z}_{n,i^\ast=0}}\boldsymbol{\lambda}_{h}=\sum_{(h_1,\ldots,h_N)\in\mathbf{Z}_{n,i^\ast=0}}\boldsymbol{\lambda}_{h}\cdot\ldots\cdot\boldsymbol{\lambda}_{h_N}=\prod_{l=1}^N\biggl[\bigl(\mathbf{\Lambda}'_{n_{{l}}}\bigr)^{-2}\Bigl(\sum_{\kappa_{{l}}=0}^{ n_{{l}}}\boldsymbol{\alpha}'_{\kappa_{{l}}}\Bigr)^2\biggr]=1.\end{equation} Analogously, in the case $\boldsymbol{l}\in N''$, we have, taking into account \eqref{positive},
\begin{multline*}\sum_{h_{{\boldsymbol{l}}}}\boldsymbol{\lambda}_{h_{{\boldsymbol{l}}}}=
\sum_{|h_{\boldsymbol{l}}|\le 2n_{\boldsymbol{l}}+2m_{\boldsymbol{l}}\Bbbk_{\boldsymbol{l}}}\boldsymbol{\lambda}_{h_{{\boldsymbol{l}}}}=\bigl(\mathbf{\Lambda}''_{n_{\boldsymbol{l}}}\bigr)^{-2}\biggl[\sum_{|\kappa_{\boldsymbol{l}}|\le n_{\boldsymbol{l}}+m_{\boldsymbol{l}}\Bbbk_{\boldsymbol{l}}}\bigl(\boldsymbol{\alpha}''_{\kappa_{\boldsymbol{l}}}\bigr)^2+
\sum_{|\kappa^\ast_{\boldsymbol{l}}-\kappa^\star_{\boldsymbol{l}}|= 1}\boldsymbol{\alpha}''_{\kappa^\ast_{\boldsymbol{l}}}\cdot \boldsymbol{\alpha}''_{\kappa^\star_{\boldsymbol{l}}}\\+\ldots+\sum_{|\kappa^\ast_{\boldsymbol{l}}-\kappa^\star_{\boldsymbol{l}}|= 2n_{\boldsymbol{l}}+2m_{\boldsymbol{l}}\Bbbk_{\boldsymbol{l}}}\boldsymbol{\alpha}''_{\kappa^\ast_{\boldsymbol{l}}}\cdot \boldsymbol{\alpha}''_{\kappa^\star_{\boldsymbol{l}}}\biggr]=
\bigl(\mathbf{\Lambda}''_{n_{\boldsymbol{l}}}\bigr)^{-2}\biggl[\sum_{|\kappa_{\boldsymbol{l}}|\le n_{\boldsymbol{l}}+m_{\boldsymbol{l}}\Bbbk_{\boldsymbol{l}}}\boldsymbol{\alpha}''_{\kappa_{\boldsymbol{l}}}\biggr]^2=1,\end{multline*}
since each of $\widetilde{\mathbf{\Psi}}$ is a linear combination of orthogonal to each other $2n_{\boldsymbol{l}}+2m_{\boldsymbol{l}}\Bbbk_{\boldsymbol{l}}+1$ sequential elements $\psi_{n_{\boldsymbol{l}},\boldsymbol{k}_{\boldsymbol{l}},\boldsymbol{s}_{\boldsymbol{l}}}$.
Therefore, on the strength of product type structure of ${\mathbf{\Psi}}_{i^\ast(d^\ast-1)\tau^\ast}$ having either $\widetilde{\mathbf{\Phi}}_{i^\ast(d^\ast-1)\tau_l^\ast}$ if $l\in N'$ or $\widetilde{\mathbf{\Psi}}_{i^\ast(d^\ast-1)\tau_l^\ast}$ if $l\in N''$ at the $l-$th place, where $N'\cup N''=N$, we obtain:
\begin{multline}\label{La''}\sum_{h\in\mathbf{Z}_{n,i^\ast\not=0}}\boldsymbol{\lambda}_{h}=\sum_{(h_1,\ldots,h_N)\in\mathbf{Z}_{n,i^\ast\not=0}}\boldsymbol{\lambda}_{h}\cdot\ldots\cdot\boldsymbol{\lambda}_{h_N}=\prod_{l\in N'}\biggl[\bigl(\mathbf{\Lambda}'_{n_{{l}}}\bigr)^{-2}\Bigl(\sum_{\kappa_{{l}}=0}^{ n_{{l}}}\boldsymbol{\alpha}'_{\kappa_{{l}}}\Bigr)^2\biggr]\\ \times\prod_{l\in N''}\biggl[\bigl(\mathbf{\Lambda}''_{n_{{l}}}\bigr)^{-2}\Bigl(\sum_{|\kappa_{{l}}|\le n_{{l}}+m_{{l}}\Bbbk_{{l}}}\boldsymbol{\alpha}''_{\kappa_{{l}}}\Bigr)^2\biggr]=1.\end{multline}

Further, we can rewrite for $f\in A_{pq}^{s}(\mathbb{R}^N,w)$ and some $h\in\mathbf{Z}_{n,i^\ast}$: \begin{multline*}\langle f(\cdot),\mathbf{\Psi}_{i^\ast(d^\ast-1)(\tau^\ast+h)}(\cdot)\rangle=
\langle f(\cdot),\mathbf{\Psi}_{i^\ast(d^\ast-1)\tau^\ast}(\cdot-h/2^{d^\ast})\rangle=\\\langle f(\cdot+h/2^{d^\ast}),\mathbf{\Psi}_{i^\ast(d^\ast-1)\tau^\ast}(\cdot)\rangle=:\langle f_{-\frac{h}{2^{d^\ast}}},\mathbf{\Psi}_{i^\ast(d^\ast-1)\tau^\ast}\rangle.\end{multline*} From this and \eqref{go},
\begin{equation}\label{drdr}\langle g,\mathbf{\Psi}_{i^\ast(d^\ast-1)\tau^\ast}\rangle=
\sum_{h\in\mathbf{Z}_{n,i^\ast}}\boldsymbol{\lambda}_{h}\langle f_{-\frac{h}{2^{d^\ast}}},\mathbf{\Psi}_{i^\ast(d^\ast-1)\tau^\ast}\rangle.\end{equation} On the strength of \eqref{arep}, representation for $f_{-\frac{h}{2^{d^\ast}}}$ can be written as follows:
\begin{align}f_{-\frac{h}{2^{d^\ast}}}(x)=
f(x+{h}/{2^{d^\ast}})=&\sum_{\nu\in\mathbb{N}_0}\int_{\mathbb{R}^N}\boldsymbol{\psi}_\nu(x+{h}/{2^{d^\ast}}-y)\,(\boldsymbol{\varphi}_\nu\ast f_{-{h}/{2^{d^\ast}}})(y)\,dy\nonumber\\ \label{harep}
=&\sum_{\nu\in\mathbb{N}_0}\int_{\mathbb{R}^N}\widetilde{{\boldsymbol{\psi}}}_\nu(x-y)\,(\widetilde{{\boldsymbol{\varphi}}}_\nu\ast f)(y)\,dz.\end{align} Here functions $\widetilde{{\boldsymbol{\psi}}}_d$ and $\widetilde{{\boldsymbol{\varphi}}}_d$ are shifted (in $-h/2^{d^\ast}$) versions of proper ${{\boldsymbol{\psi}}}_d$ and ${{\boldsymbol{\varphi}}}_d$, inheriting their properties (see \cite[Theorem 1.6]{R}). Thus, \eqref{harep} can be viewed as local representation of the original element $f\in A_{pq}^{s}(\mathbb{R}^N,w)$ with help of shifted with respect to $\{{\boldsymbol{\psi}}\}_{d\in\mathbb{N}_0}$ and $\{{\boldsymbol{\varphi}}\}_{d\in\mathbb{N}_0}$ functions $\{\widetilde{\boldsymbol{\psi}}\}_{d\in\mathbb{N}_0}$ and $\{\widetilde{\boldsymbol{\varphi}}\}_{d\in\mathbb{N}_0}$ preserving the same properties. This coincides with \eqref{arep} for $f$ and allows to write \eqref{drdr} in the following form:
\begin{equation*}
\langle g,\mathbf{\Psi}_{i^\ast(d^\ast-1)\tau^\ast}\rangle=\langle f,\mathbf{\Psi}_{i^\ast(d^\ast-1)\tau^\ast}\rangle
\cdot \sum_{h\in\mathbf{Z}_{n,i^\ast}}\boldsymbol{\lambda}_{h}=:\langle f,\mathbf{\Psi}_{i^\ast(d^\ast-1)\tau^\ast}\rangle\end{equation*} (see \eqref{La'} and \eqref{La''}).
This could be extended to any finite linear combination of $\mathbf{\Psi}_{i^\ast(d^\ast-1)\tau^\ast}$. 

{Both distributions $f$ and $g$ are locally contained in $B^\sigma_{pp}(\mathbb{R}^N)$ for any $\sigma<s-N(\boldsymbol{r}_0-1)/p$. This follows} {from \cite[(21)]{Ma} and, in addition, from elementary embeddings between weighted Besov} {and Triebel--Lizorkin spaces if $A_{pq}^{s,w}(\mathbb{R})=F_{pq}^{s,w}(\mathbb{R}^N)$ (see e.g. \cite[\S~2.3]{R}).} Any $\phi\in C_0^{\infty}(\mathbb{R}^N)$ has the unique $L_2(\mathbb{R}^N)$ wavelet representation{, which converges in the dual space of $B^\sigma_{pp}(\mathbb{R}^N)$ if we choose $\sigma$} {so that $n-1>\max\{\sigma_p-\sigma,\sigma\}$ for $A_{pq}^{s,w}(\mathbb{R}^N)=B_{pq}^{s,w}(\mathbb{R}^N)$ and $n-1>\max\{\sigma_{p,q}-\sigma,\sigma\}$ in the case $A_{pq}^{s,w}(\mathbb{R}^N)=F_{pq}^{s,w}(\mathbb{R}^N)$ \cite{TBan}. This} implies $\langle g,\phi\rangle=\langle f,\phi\rangle$ for all $\phi\in C_0^{\infty}(\mathbb{R}^N)$, that is $g=f$.

By the above, $f\in\mathscr{S}'_e(\mathbb{R}^N)$ belongs to $A_{pq}^{s,w}(\mathbb{R}^N)$ if and only if $$f=\sum_{\tau\in\mathbb{Z}^N} \lambda_{00\tau}\mathbf{\Phi}_\tau+
\sum_{d\in\mathbb{N}}\sum_{i=1}^{2^N-1}\sum_{\tau\in\mathbb{Z}^N}\lambda_{id\tau}2^{-dN/2}\mathbf{\Psi}_{i(d-1)\tau},
$$ and $\lambda\in a_{pq}^{s,w}$ (or $\boldsymbol{\lambda}\in \boldsymbol{a}_{pq}^{w}$). This representation is unique with $\lambda_{00\tau}=k_{00\tau}(f)$, $\tau\in\mathbb{Z}^N$, and $\lambda_{id\tau}=k_{id\tau}(f)$, ${i=1,\ldots,2^N-1}$, ${d\in\mathbb{N}}$, $\tau\in\mathbb{Z}^N$, and the estimates \eqref{isom} hold. The uniqueness of the coefficients {implies} that $I$ is the required isomorphism. \end{proof}

\begin{remark}\label{Rrr}{\rm Theorem \ref{main} is using spline wavelet bases \eqref{ForRepr'} generated by \eqref{wavelets_main}. Elements of \eqref{wavelets_main} have explicit forms, and their main components $B-$splines satisfy dif\-fe\-ren\-ti\-a\-tion property \eqref{diff}. This fact makes our result applicable to the study of integration or differentiation operators of natural orders in $A_{pq}^{s,w}(\mathbb{R}^N)\cap L_1^\textrm{loc}(\mathbb{R}^N)$.

Another decomposition theorem in function spaces $ A_{pq}^{s,w}(\mathbb{R}^N)$ with local Muckenhoupt weights was performed in \cite[Theorem 6.2]{WBan} in terms of Daubechies wavelets \cite{Dau}. Our main theorem is analogous to that result, which is valid under the same assumptions \eqref{condB} and \eqref{condF} on parameters. But our Theorem \ref{main} is using dictionary of Battle--Lemari\'{e} spline wavelet systems instead of that of Daubechies type as in \cite[Theorem 6.2]{WBan}. 

Theorem \ref{main} is an extension of spline wavelet decompositions from \cite[\S~2.5]{Tr5} by H. Triebel in unweighted spaces $A_{pq}^{s}(\mathbb{R}^N)$, upto weighted $A_{pq}^{s,w}(\mathbb{R}^N)$ with $w\in\mathscr{A}_\infty^\loc$ (see also \cite[Proposition 4.2]{JMAA}). Theorem \ref{main} generalises also \cite[Theorem 4.7]{RMC}, which works for function spaces $A_{pq}^{s}(\mathbb{R}^N,w)$ with $w\in\mathscr{A}_\infty$.

Observe that in the unweighted case the restrictions on $n_0$ can be taken less strong than in \eqref{condB} and \eqref{condF} (see e.g. \cite[\S~2.5]{Tr5} or \cite[Proposition 4.2]{JMAA}). In particular, even the Haar system can be used for decomposing Besov spaces $B_{pq}^{s}(\mathbb{R}^N)$ in the range $1/p-1<s< \min\{1/p,1\}$. For Triebel--Lizorkin spaces $F_{pq}^{s}(\mathbb{R}^N)$ conditions on $s$ look a bit more complicated \cite[\S~2.5]{Tr5}. By recent papers of G. Garrig\'{o}s, A. Seeger and T. Ullrich one can see that those restrictions on smoothness $s$ are optimal when decomposing $A_{pq}^{s}(\mathbb{R}^N)$ by the Haar systems \cite{GSU1, GSU2, GSU3}.}
\end{remark}



\noindent
Ushakova Elena Pavlovna

\smallskip
\noindent
V.A. Trapeznikov Institute of Control Sciences of Russian Academy of Sciences

\noindent
65 Profsoyuznaya St., Moscow 117997, Russia.

\smallskip
\noindent
Steklov Mathematical Institute of Russian Academy of Sciences

\noindent
8 Gubkina St., Moscow 119991, Russia.

\smallskip
\noindent
Computing Center of Far--Eastern Branch of Russian Academy of Sciences

\noindent
65 Kim Yu Chena St., Khabarovsk 680000, Russia.

\smallskip
\noindent
E-mail: elenau@inbox.ru

\end{document}